\documentclass[final,3p,times]{elsarticle}

\usepackage[utf8]{inputenc}
\usepackage{amsmath}
\usepackage{amsfonts}
\usepackage{amssymb}
\usepackage{amsthm}
\usepackage{graphicx}
\usepackage{easybmat}
\usepackage{pifont} 
\usepackage{enumitem}
\usepackage{float}
\usepackage{subcaption}
\usepackage{url}

\usepackage{mathtools}
\mathtoolsset{showonlyrefs}
\biboptions{sort&compress}

\usepackage{color}
\def\Red#1{\textcolor{red}{#1}}

\usepackage{ragged2e}

\usepackage{natbib}

\newtheorem{theorem}{Theorem}

\newtheorem{definition}{Definition}
\newtheorem{example}{Example}
\newtheorem{lemma}{Lemma}

\newtheorem{remark}{Remark}

\def\sign{\hskip2pt{\rm sign}\hskip0pt}

\journal{ArXiv}

\begin{document}

\begin{frontmatter}



\title{A class of robust consensus algorithms with predefined-time convergence under switching topologies
\footnote{\Red{This is the preprint version of the accepted Manuscript: R. Aldana-López, D. Gómez-Gutiérrez, M. Defoort, J. D. Sánchez-Torres and A.~J.~Mu\~noz-V\'azquez, “A class of robust consensus algorithms with predefined-time convergence under switching topologies”, International Journal of Robust and Nonlinear Control, 2019, ISSN: 1099-1239. DOI. 10.1002/rnc.4715.
Please cite the publisher's version. For the publisher's version and full citation details see:
\url{http://dx.doi.org/10.1002/rnc.4715}
}}
}

\author[label0]{R.~Aldana-L\'opez}
\ead{rodrigo.aldana.lopez@gmail.com}

\author[label0,label1]{David~Gómez-Gutiérrez\corref{cor1}}
\ead{David.Gomez.G@ieee.org}
\cortext[cor1]{Corresponding Author.}
\author[label4]{M.~Defoort}
\ead{michael.defoort@univ-valenciennes.fr}
\author[label3]{J.~D.~S\'anchez-Torres}
\ead{dsanchez@iteso.mx}

\author[UAC]{A.~J.~Mu\~noz-V\'azquez}\ead{aldo.munoz.vazquez@gmail.com}

\address[label0]{Multi-agent autonomous systems lab, Intel Labs, Intel Tecnología de M\'exico, Av. del Bosque 1001, Colonia El Bajío, Zapopan, 45019, Jalisco, M\'exico.}
\address[label1]{Tecnologico de Monterrey, Escuela de Ingenier\'ia y Ciencias, Av. General Ram\'on Corona 2514, Zapopan, 45201, Jalisco, M\'exico.}

\address[label3]{Research Laboratory on Optimal Design, Devices and Advanced Materials -OPTIMA-, Department of Mathematics and Physics, ITESO, Perif\'erico Sur Manuel G\'omez Mor\'in 8585 C.P. 45604, Tlaquepaque, Jalisco, M\'exico.}
\address[label4]{LAMIH, CNRS UMR 8201, Univ. Valenciennes, Valenciennes 59313, France.}
\address[UAC]{CONACYT -- School of Engineering, Autonomous University of Chihuahua, Chihuahua, M\'exico.}

\begin{abstract}
This paper addresses the robust consensus problem under switching topologies. Contrary to existing methods, the proposed approach provides decentralized protocols that achieve consensus for networked multi-agent systems in a predefined time. Namely, the protocol design provides a tuning parameter that allows setting the convergence time of the agents to a consensus state.  An appropriate Lyapunov analysis exposes the capability of the current proposal to achieve predefined-time consensus over switching topologies despite the presence of bounded perturbations. Finally, the paper presents a comparison showing that the suggested approach subsumes existing fixed-time consensus algorithms and provides extra degrees of freedom to obtain predefined-time consensus protocols that are less over-engineered, i.e., the difference between the estimated convergence time and its actual value is lower in our approach. Numerical results are given to illustrate the effectiveness and advantages of the proposed approach.
\end{abstract}


\end{frontmatter}

\section{Introduction}

The problem of cooperative control for multi-agent systems has received significant attention due to its wide range of applications in different fields such as biology, power grid, robotics, etc. In cooperative control, a fundamental challenge is the consensus problem, whose objective is to design decentralized protocols such that all the agents, in the network, interact with their neighbors to reach a common value~\cite{Olfati-Saber2007}, called the consensus equilibrium. Many decentralized consensus-based methods have been applied, for instance, to flocking~\cite{Olfati2006}, formation control~\cite{Oh2015,Ren2007,Li2013} and distributed resource allocation~\cite{Xu2017,Xu2017b}. Two consensus problems have been mainly addressed, the leader-follower consensus problem~\cite{Defoort2015,8543491}, where the agents along the network converge to the state of the leader (real or virtual) which communicates its state only to a subset of agents; and the leaderless consensus problem~\cite{Zuo2014,Ning2017b,Zuo2016}, where the consensus equilibrium is a function of the initial conditions of the agents, for instance achieving consensus to the average value~\cite{Olfati2006}, the average min-max value~\cite{Cortes2006,Li2014}, the median value~\cite{Franceschelli2017}, and the maximum or minimum value~\cite{Liu2015} of the agents' initial conditions.

Many existing consensus algorithms focus on asymptotic convergence (i.e. the settling time is infinite) \cite{ren2005consensus} and finite-time convergence (i.e. the settling time estimate is finite but depends on the initial conditions) \cite{Wang2010,zhao2018edge}. It is clear that a finite settling time is useful for several applications (a network of clusters, manufacturing systems, etc.). However, the initial conditions are usually unknown due to the decentralized architecture. Therefore, recently, there has been a great deal of attention in the research community on algorithms which solve the decentralized consensus problem in a finite-time with uniform convergence with respect to the initial value of the agents (see. e.g~\cite{Gomez-Gutierrez2018} and the references therein). 

Two main approaches have been proposed to address this problem. The first one is based on the fixed-time stability theory of autonomous systems proposed in~\cite{Polyakov2012}. Based on this class of systems, fixed-time consensus algorithms have been proposed either based on the homogeneity theory developed in~\cite{Andrieu2008,Polyakov2016}, as in~\cite{Gomez-Gutierrez2018}, the settling-time estimation provided in~\cite{Polyakov2012}, as in~\cite{Zuo2014a,Parsegov2013} or the subclass of fixed-time stable systems, with an upper bound of the settling time that is less conservative than the one in~\cite{Polyakov2012}, presented in~\cite{Parsegov2012}, as in~\cite{Parsegov2013,Zuo2014,Defoort2015,Ning2017b,Wang2017b}. However, in spite of the advantages in terms of settling time estimation, this approach does not enable to easily arbitrarily preassigned the settling time. For this case, the consensus problem for agents with single integrator dynamics and affected by disturbances has been addressed in~\cite{Zuo2016,Ning2017b}.

The second approach is based on time-varying consensus protocols and provides consensus protocols with predefined-time convergence, where the convergence time is set a priori as a parameter of the protocol~\cite{Yong2012,Liu2018,Wang2017,Wang2018,Colunga2018b,Zhao2018}. In these works, non-conservative upper-bound estimates of the settling time are provided. However, this approach presents some drawbacks. First, they are based on a common time-varying gain that is applied to each node on the network, which requires a common time-reference along the network. Second, this time-varying gain becomes singular at the predefined-time, such as methods based on the so-called time base generators~\cite{Morasso1997} see for instance~\cite{Yong2012}. Third, methods, such as~\cite{Liu2018,Zhao2018}, are based on a time-varying gain that is a piecewise-constant function that produces Zeno phenomenon~\cite{Zhang2001}, even in the unperturbed case. 
Additionally, the robustness, against external disturbances, of the methods proposed in~\cite{Yong2012,Liu2018,Wang2017,Wang2018,Zhao2018} has not been analyzed.

In this paper, we address the robust predefined-time leaderless consensus problem for the cases where a common time reference is not available. Thus, we focus on developing autonomous protocols. Our approach is based on recent results on the upper bound estimation for the settling time of a class of fixed-time stable systems presented in~\cite{Aldana-Lopez2018} to propose a broader class of consensus protocols for perturbed first-order agents with a fixed-time convergence whose convergence upper bound is set a priori as a parameter of the system. In our opinion,~\cite{Ning2017b} is the closest approach in the literature, as it presents robust fixed-time consensus protocols that can be designed to satisfy time constraints. We will show that the results in~\cite{Ning2017b} and other fixed-time consensus protocols, such as~\cite{Zuo2014,Zuo2014a,Wang2017b}, are subsumed by our approach since our results allow more flexibility in the parameter selection to obtain, for instance, predefined-time consensus protocols where the slack between the real-convergence time and the estimated upper bound of the convergence time is lower. Contrary to previous fixed-time consensus protocols~\cite{Parsegov2013,Zuo2014,Zuo2014a,Wang2017b,Zuo2016,Ning2017b} our approach allows to straightforwardly specify the upper bound of the convergence time as a parameter of the consensus protocol. We presents two robust consensus protocols. The first one is the one that is computationally simpler; we show that it presents predefined-time convergence under static networks and fixed-time convergence under switched dynamic networks. The second one, in the absence of disturbances, converges to a consensus state that is the average of the agents' initial conditions. This algorithm is shown to be a robust predefined-time consensus algorithm for static and dynamic networks arbitrarily switching among connected graphs.

The rest of the manuscript is organized as follows. Section~\ref{Sec.Preliminaries} introduces the preliminaries on graph theory, finite-time stability and predefined-time stability. Section~\ref{Sec:Main} presents two approaches to address the robust consensus problem for dynamic networks. Section~\ref{Sec:Disc} presents a comparison with~\cite{Ning2017b} showing that having a broader class of predefined-time consensus algorithms provides greater flexibility to improve the performance of the consensus algorithms. This section also presents a qualitative comparison of both protocols with respect to other fixed-time consensus protocols previously proposed in the literature highlighting the contribution of our approach. Finally, Section~\ref{Sec:Concl} presents some concluding remarks.

\section{Preliminaries}\label{Sec.Preliminaries}

\subsection{Graph Theory}
\label{SubSec.GraphTheory}
The following preliminaries on graph theory are from~\cite{godsil2001}. In this paper, we will focus only on undirected graphs.

\begin{definition}
A graph $\mathcal{X}$ consists of a set of vertices $\mathcal{V}(\mathcal{X})$ and a set of edges $\mathcal{E}(\mathcal{X})$ where an edge is an unordered pair of distinct vertices of $\mathcal{X}$. 
\end{definition}

Writing $ij$ denotes an edge and $j\sim i$ denotes that the vertex $i$ and vertex $j$ are adjacent or neighbors, i.e., there exists an edge $ij$. The set of neighbors of $i$ in the graph $\mathcal{X}$ is expressed by $\mathcal{N}_i(\mathcal{X})=\{j:ji\in \mathcal{E}(\mathcal{X})\}$.


\begin{definition}
A path from $i$ to $j$ in a graph is a sequence of distinct vertices starting with $i$ and ending with $j$ such that consecutive vertices are adjacent. If there is a path between any two vertices of the graph $\mathcal{X}$ then $\mathcal{X}$ is said to be connected. Otherwise, it is said to be disconnected. 
\end{definition}

\begin{definition}
A weighted graph is a graph together with a weight function $\mathcal{W}:\mathcal{E}(\mathcal{X})\to \mathbb{R}_{+}$.
\end{definition}

\begin{definition}
Let $\mathcal{X}$ be a weighted graph such that $ij\in\mathcal{E}(\mathcal{X})$ has weight $a_{ij}$ and let $n=|\mathcal{V}(\mathcal{X})|$. Then, the adjacency matrix $A$ is an $n\times n$ matrix where $A=[a_{ij}]$.
\end{definition}

\begin{definition}
Let $\mathcal{X}$ be a graph, the Laplacian of $\mathcal{X}$ is denoted by $\mathcal{Q}(\mathcal{X})$ (or simply $\mathcal{Q}$ when the graph is clear from the context) and is defined as $\mathcal{Q}(\mathcal{X})=\Delta(\mathcal{X})-A$ where $\Delta(\mathcal{X})=\text{diag}(d_1\cdots,d_n)$ with $d_i=\sum\limits_{j\in\mathcal{N}_i(\mathcal{X})}a_{ij}$. 
\end{definition}

\begin{remark}\label{remark:Lapl_factorization} For the Laplacian $\mathcal{Q}(\mathcal{X})$, there exists a factorization  $\mathcal{Q}(\mathcal{X})=D(\mathcal{X})D(\mathcal{X})^T$ ($D(\mathcal{X})$ is known as the incidence matrix of $\mathcal{X}$~\cite{godsil2001}) where $D(\mathcal{X})$ is an $\vert \mathcal{V}(\mathcal{X})\vert \times \vert \mathcal{E}(\mathcal{X})\vert$ matrix, such that if $ij\in \mathcal{E}(\mathcal{X})$ is an edge with weight $a_{ij}$ then the column of $D$ corresponding to the edge $ij$ has only two nonzero elements: the $i-$th element is equal to $\sqrt{a_{ij}}$ and the $j-$th element is equal to $-\sqrt{a_{ij}}$. Clearly, the incidence matrix $D(\mathcal{X})$, satisfies $\mathbf{1}^TD(\mathcal{X})=0$.
The Laplacian matrix $\mathcal{Q}(\mathcal{X})$ is a positive semidefinite and symmetric matrix. Thus, its eigenvalues are all real and non-negative.
\end{remark}

When the graph $\mathcal{X}$ is clear from the context we omit $\mathcal{X}$ as an argument. For instance we write $Q$, $D$, etc to represent the Laplacian, the incidence matrix, etc.

\begin{lemma}~\cite{godsil2001}
\label{lemma:Lambda2}
Let $\mathcal{X}$ be a connected graph and $\mathcal{Q}$ its Laplacian. The eigenvalue $\lambda_1(\mathcal{Q})=0$ has algebraic multiplicity one with eigenvector $\mathbf{1}=[1\ \cdots\ 1]^T$. The smallest nonzero eigenvalue of $\mathcal{Q}$, denoted by $\lambda_2(\mathcal{Q})$ satisfies $\lambda_2(\mathcal{Q})=\underset{x\perp \mathbf{1},x\neq 0}{\min}\dfrac{x^T \mathcal{Q}x}{x^Tx}$. 
\end{lemma}

It follows from Lemma~\ref{lemma:Lambda2} that for every $x\bot\mathbf{1}$, $x^T\mathcal{Q}x\geq \lambda_2(\mathcal{Q}) \Vert x \Vert_2^2>0$. $\lambda_2(\mathcal{Q})$ is known as the algebraic connectivity of the graph $\mathcal{X}$. The distance between two distinct nodes $i$ and $j$ is the shortest path length between them. The diameter of $\mathcal{X}$ is the longest distance between two distinct vertices. 

For a path graph and cycle graph of $n$ nodes, $P_n$ and $C_n$, $\lambda_2$ can be computed as $\lambda_2(P_n)=2-2\cos\left(\pi/n\right)$ and  $\lambda_2(C_n)=2-2\cos\left(2\pi/n\right)$, respectively. For a star graph $S_n$, $\lambda_2(S_n)=1$. Graphs that are more connected have a larger $\lambda_2$. A path graph $P_n$ is the ``most nearly disconnected" connected graph of $n$ nodes.
\begin{definition}
A switched dynamic network $\mathcal{X}_{\sigma(t)}$ is described by the ordered pair $\mathcal{X}_{\sigma(t)}=\langle\mathcal{F},\sigma\rangle$ where $\mathcal{F}=\{\mathcal{X}_1,\ldots,\mathcal{X}_m\}$ is a collection of graphs having the same vertex set and $\sigma:[t_0,\infty)\rightarrow \{1,\ldots m\}$ is a switching signal determining the topology of the dynamic network at each instant of time. 
\end{definition}

In this paper, we assume that $\sigma(t)$ is generated exogenously and that there is a minimum dwell time between consecutive switchings in such a way that Zeno behavior in network's dynamic is excluded, i.e., there is a finite number of switchings in any finite interval.



\subsection{On finite-time, fixed-time and predefined-time stability}
Consider the system
\begin{equation} \label{eq:sys}
\dot{x}=f(x;\rho), 
\end{equation}
where $x\in\mathbb{R}^n$ is the system state, the vector $\rho\in\mathbb{R}^b$ stands for the parameters of system~\eqref{eq:sys} which are assumed to be constant, i.e., $\dot{\rho}=0$. Furthermore, there is no limit for the number of parameters, so $b$ can take any value in the natural number set $\mathbb{N}$. The function $f:\mathbb{R}^n\rightarrow\mathbb{R}^n$ is nonlinear, and the origin is assumed to be an equilibrium point of system~\eqref{eq:sys}, so $f(0;\rho)=0$. The initial condition of this system is $x_0=x(0)\in\mathbb{R}^n$.

\begin{definition}\cite{Bhat2000} \label{def:finite} The origin of~\eqref{eq:sys} is \textit{globally finite-time stable} if it is globally asymptotically stable and any solution $x(t,x_0)$ of \eqref{eq:sys} reaches the equilibrium point at some finite time moment, i.e., $\forall t\geq T(x_0):x(t,x_0)=0$, where $T:\mathbb{R}^{n}\rightarrow \mathbb{R}_{+}\cup \{0\}$ is called the \textit{settling-time function}.
\end{definition}

\begin{definition}\cite{Polyakov2012} \label{def:fixed} The origin of~\eqref{eq:sys} is \textit{fixed-time stable} if it is globally finite-time stable and the settling-time function is bounded, i.e. $\exists T_{\text{max}}>0:\forall x_0\in\mathbb{R}^{n}:T(x_0)\leq T_{\text{max}}$.
\end{definition}

To address the fixed-time stability analysis, one approach is based on the homogeneity theory.
\begin{definition}
\label{Def.Homogeneous}
\cite{Bhat2005}
A function $g:\mathbb{R}^n\rightarrow\mathbb{R}$ is homogeneous of degree $l$ with respect to the standard dilation if and only if
 $$g(\lambda x)=\lambda^l g(x)\ \ \text{for all } \lambda>0.$$
A vector field $f(x)$, with $x\in\mathbb{R}^n$, is homogeneous of degree $d$ with respect to the the standard dilation if
$$
f(x)=\lambda^{-(d+1)} f(\lambda x) \ \ \text{for all } \lambda>0.
$$ 
\end{definition}

\begin{definition}
\label{Def:Bilimit}
\cite{Andrieu2008,Polyakov2016}
A function $g:\mathbb{R}^n\rightarrow\mathbb{R}$, such that $g(0)=0$, is homogeneous in the $\lambda_0-$limit with degree $d_{\lambda_0}$ with respect to the standard dilation if $g_{\lambda_0}:\mathbb{R}^n\rightarrow \mathbb{R}$, defined as
$$
g_{\lambda_0}(x)=\lim_{\lambda\rightarrow\lambda_0}\lambda^{-d_{\lambda_0}}g(\lambda x),
$$
is homogeneous of degree $d_{\lambda_0}$ with respect to the standard dilation.\\
A vector field $f:\mathbb{R}^n\rightarrow\mathbb{R}^n$ is said to be homogeneous in the $\lambda_0-$limit with degree $d_{\lambda_0}$ with respect to the standard dilation if the vector field $f_{\lambda_0}:\mathbb{R}^n\rightarrow \mathbb{R}^n$, defined as
\begin{equation}
\label{Eq:Hom}
f_{\lambda_0}(x)=\lim_{\lambda\rightarrow\lambda_0}\lambda^{-(d_{\lambda_0}+1)}f(\lambda x),
\end{equation}
is homogeneous of degree $d_{\lambda_0}$ with respect to the standard dilation.
\end{definition}

A characterization of fixed-time stability, based on the homogeneity theory, is given in the following theorem.

\begin{theorem}
\label{Th:Fixed}
\cite{Andrieu2008,Polyakov2016}
Let the vector field $f:\mathbb{R}^n\rightarrow\mathbb{R}^n$ be homogeneous in the $0-$limit with degree $d_0<0$ and homogeneous in the $+\infty$-limit with degree $d_\infty>0$ (if both conditions are satisfied it is said that the vector field f(x) is homogeneous in the bi-limit). If for the dynamic systems $\dot{x}=-f(x)$, $\dot{x}=-f_0(x)$ and $\dot{x}=-f_\infty(x)$ the origin is globally asymptotically stable (where $f_0$ and $f_{\infty}$ are obtained from~\eqref{Eq:Hom} with $\lambda_0=0$ and $\lambda_0=+\infty$, respectively), then the origin of $\dot{x}=-f(x)$ is a globally fixed-time stable equilibrium.
\end{theorem}

However, despite the advantages of having a finite settling time which allows uniform convergence with respect to the initial conditions, the homogeneity based approach does not enable to easily arbitrarily preassigned the settling time. For some applications such as state estimation, dynamic optimization, consensus of cluster networks, among others, it would be convenient that the trajectories of system~\eqref{eq:sys} reach the origin within a time $T_c\in\mathcal{T}$, which can be defined in advance as a function of the system parameters, i.e., $T_c=T_c(\rho)$. This motivates the following definitions.


\begin{definition}\cite{Sanchez-Torres2014} \label{def:min_bound} Let the origin be fixed-time stable for system~\eqref{eq:sys}. The set of all the bounds of the settling-time function is defined as \[\mathcal{T}=\left\lbrace T_\text{max}\in\mathbb{R}_+ : T(x_0)\leq T_\text{max}, \; \forall \, x_0 \in \mathbb{R}^n \right\rbrace.\]
\end{definition}

\begin{definition} \cite{Sanchez-Torres2015,Sanchez-Torres2018} \label{def:predefined} For the  parameter vector  $\rho$ of system~\eqref{eq:sys} and a constant $T_c:=T_c(\rho)>0$, the origin of~\eqref{eq:sys} is said to be \textit{predefined-time stable} 
if it is fixed-time stable and the settling-time function $T:\mathbb{R}^n\rightarrow\mathbb{R}$ is such that \[T(x_0)\leq T_c, \quad \forall x_0\in\mathbb{R}^n.\] $T_c$ is called a \textit{predefined time}.
\end{definition}

\begin{remark}\label{rem:uncertain} It would be desirable to choose $T_c=T_c(\rho)$ not only as a bound of the settling-time function $T_c\in\mathcal{T}$, but as the least upper bound, i.e., $T_c=\min \mathcal{T} = \sup_{x_0 \in \mathbb{R}^n} T(x_0)$. However, this selection requires complete knowledge about the system, compromising its application to decentralized systems.
\end{remark}

Let us recall some important results concerning predefined-time stability. 

\begin{theorem}\cite[Theorem~1]{Aldana-Lopez2018}
\label{th:tf_poly}
Consider system
\begin{equation}
\label{Eq:HomFixedPoly}
\dot{x} = -(\alpha|x|^{p} + \beta|x|^{q})^k \sign{x}, \ \ x(0)=x_0,
\end{equation}
where $x\in\mathbb{R}$, $t\in\left[0,+\infty\right)$. The parameters of the system are the real numbers $\alpha,\beta,p,q,k>0$ which satisfy the constraints $kp<1$ and $kq>1$.
Let $\rho$ be the parameter vector $\rho=\left[\alpha \, \beta \, p \, q \, k\right]^T\in\mathbb{R}^5$ of~\eqref{Eq:HomFixedPoly}. Then, the origin $x=0$ of system~\eqref{Eq:HomFixedPoly} is fixed-time stable and the settling time function satisfies $T_f=\gamma(\rho)$, where 
\begin{equation}
    \label{Eq:MinUpperEstimate}
    \gamma(\rho)=\frac{\Gamma \left(\frac{1-k p}{q-p}\right) \Gamma \left(\frac{k q-1}{q-p}\right)}{\alpha^{k}\Gamma (k) (q-p)}\left(\frac{\alpha}{\beta}\right)^{\frac{1-k p}{q-p}},
    \end{equation}
and $\Gamma(\cdot)$ is the Gamma function defined as $\Gamma(z)=\int_0^{+\infty} e^{-t}t^{z-1}dt$~\cite[Chapter~1]{Bateman1953}.
\end{theorem}

Using Theorem \ref{th:tf_poly}, one can obtain the following Lyapunov based characterization of predefined-time stable systems.  


\begin{theorem}\cite[Theorem~3]{Aldana-Lopez2018}\label{thm:weak_pt}
Consider the nonlinear system
\begin{equation}\label{Eq:NonSystemDyn}
    \dot{x}=f(x;\rho), \ \ x(0)=x_0,
\end{equation}
where $x\in\mathbb{R}^n$ is the system state, the vector $\rho\in\mathbb{R}^b$ stands for the system parameters which are assumed to be constant. The function $f:\mathbb{R}^n\rightarrow\mathbb{R}^n$ is such that $f(0;\rho)=0$. Assume there exists a continuous radially unbounded function $V:\mathbb{R}^n\to\mathbb{R}$ such that:
\begin{align*}
V(0) &=0, \\
V(x) &>0, \qquad \forall x\in\mathbb{R}^n\setminus\{0\},
\end{align*}
and the derivative of $V$ along the trajectories of~\eqref{Eq:NonSystemDyn} satisfies
 \begin{equation}\label{eq:dV_weak}
 \mathcal{D}^{+}V(x)\leq-\frac{\gamma(\rho)}{T_c}\left(\alpha V(x)^p+\beta V(x)^q\right)^k, \qquad \forall x\in\mathbb{R}^n\setminus\{0\}.
 \end{equation}  
where $\alpha,\beta,p,q,k>0$, $kp<1$, $kq>1$ and $\gamma$ is given in~\eqref{Eq:MinUpperEstimate} and $\mathcal{D}^{+}V$ is the upper right-hand Dini derivative of $V(x)$ \cite{Kannan2012}.

Then, the origin of \eqref{Eq:NonSystemDyn} is predefined-time-stable with predefined time $T_c$. 
\end{theorem}
\section{Main results}
\label{Sec:Main}
\subsection{Problem statement}

Consider a multi-agent system consisting of $n$ agents with first-order dynamics given by
\begin{equation}\label{Eq:AgentDyn}
\dot x_i(t)=u_i(t)+d_i(t), \qquad i=1,\ldots,n,
\end{equation}
where $x_i\in \mathbb{R}$ and $u_i\in \mathbb{R}$ are the system state and control input respectively. Term $d_i$ represents external perturbations and is assumed to be bounded by a known positive constant $L_i$, i.e.
\begin{equation}\label{Eq:perturbation}
|d_i(t)| \leq L_i
\end{equation}
The communication topology between agents is represented by the graph (which could be dynamic) $\mathcal{X}_{\sigma(t)}=\langle\mathcal{F},\sigma\rangle$ where $\mathcal{F}=\{\mathcal{X}_1,\ldots,\mathcal{X}_m\}$ is a collection of graphs having the same vertex set and $\sigma:[0,\infty)\rightarrow \{1,\ldots m\}$ is the switching signal determining the topology of the dynamic network at each instant of time. 

The control objective is to design decentralized consensus protocols $u_i$ ($i=1,\ldots,n$), based on available local information, such that the state of all the agents converges to an equilibrium, known as a consensus state $x^*$, in a predefined-time $T_c$, regardless of the initial conditions $x_i(0)$ of the agents and the disturbance signals $d_i(t)$ affecting each agent, i.e.
\begin{equation}\label{Eq:control_objective}
x_1(t)=\cdots=x_n(t)=x^*, \qquad \forall t \geq T_c
\end{equation}

\begin{remark}
If $x^*=\frac{1}{n}\sum_{i=1}^{n}x_i(0)$, the control objective is called as the predefined-time average consensus problem. 
\end{remark}

\begin{remark}
There exist methods for predefined-time consensus in the literature. However, they are based on time-varying gains~\cite{Yong2012,Liu2018,Wang2017,Wang2018,Colunga2018b,Zhao2018}, since the same gain must be applied to all agents, a common time-reference is needed along the network. Moreover, this time-varying gain becomes singular at the predefined-time, see for instance~\cite{Yong2012} or as in \cite{Liu2018,Zhao2018} where it may produce a Zeno behavior (infinite switching in a finite interval). These drawbacks are not present in our approach.
\end{remark}

\subsection{Robust predefined-time consensus algorithms}

In this section, we propose and analyze two fixed-time consensus protocols, which have the following form 
\begin{equation}
\label{Eq:ProtocolA}
u_i = \kappa_i\left[(\alpha|e_{i}|^{p} + \beta|e_{i}|^{q})^k+\zeta\right] \sign{e_{i}}, \ \ e_{i}=\sum_{j\in\mathcal{N}_i(\mathcal{X}_{\sigma(t)})} a_{ij} (x_j(t)-x_i(t))
\end{equation}
and
\begin{equation}
\label{Eq:ProtocolB}
u_i = \kappa_i \sum_{j\in\mathcal{N}_i(\mathcal{X}_{\sigma(t)})} \sqrt{a_{ij}}\left[(\alpha|e_{ij}|^{p} + \beta|e_{ij}|^{q})^k +\zeta\right]\sign{e_{ij}}, \ \ e_{ij}=\sqrt{a_{ij}}(x_j(t)-x_i(t))
\end{equation}
where $\alpha,\beta,p,q,k>0$, $kp<1$ and $kq>1$. The constant $\zeta$ and $\kappa_i$ will be designed later to guarantee the convergence in a predefined-time $T_c$, even in the presence of disturbances. 

\begin{remark}
Let $\phi(\cdot)=\left[(\alpha|\cdot|^{p} + \beta|\cdot|^{q})^k+\zeta\right] \sign{\cdot}$. On the one hand, notice that the consensus protocol~\eqref{Eq:ProtocolA} is computationally simpler than consensus protocol~\eqref{Eq:ProtocolB} since at each time instance each agent only requires a single computation of the nonlinear function $\phi(\cdot)$ whereas for protocol ~\eqref{Eq:ProtocolB} an agent requires to perform one computation of $\phi(\cdot)$ for each neighbor. On the other hand, we will show that~\eqref{Eq:ProtocolA} is a predefined-time consensus protocol for static networks and a fixed-time consensus protocol for dynamic networks 
whereas control~\eqref{Eq:ProtocolB} is a predefined-time consensus protocol for static and dynamic networks
\end{remark}

\subsubsection{Convergence analysis under consensus protocol (\ref{Eq:ProtocolA})}


Let $e=[e_1 \ \cdots \ e_n]^T$, where $e_i$ is given in~\eqref{Eq:ProtocolA}, then $e$ can be written as $e=\mathcal{Q}(\mathcal{X}_{\sigma(t)})x$ with $x=[x_1 \ \cdots \ x_n]^T$. Notice that, the dynamic of the network under the consensus algorithm \eqref{Eq:ProtocolA} is given by
\begin{equation}
\label{ConsensusDynamicA}
\dot{x}=-\Phi(\mathcal{Q}(\mathcal{X}_{\sigma(t)})x)+\Delta(t),
\end{equation}
where, for $z=[z_1 \ \cdots \ z_n]^T\in\mathbb{R}^n$, the function $\Phi:\mathbb{R}^n\rightarrow \mathbb{R}^n$ is defined as
\begin{equation}
\label{Eq:Fp}    
\Phi(z)=
\left[
\begin{array}{c}
\kappa_1\left[(\alpha|z_1|^{p} + \beta|z_1|^{q})^k+\zeta\right] \sign{z_1} \\ \vdots \\ \kappa_n\left[(\alpha|z_n|^{p} + \beta|z_n|^{q})^k+\zeta\right] \sign{z_n}
\end{array}
\right],
\end{equation}
and $\Delta(t)=[d_1(t) \ \cdots \ d_n(t)]^T$ with $\|\Delta(t)\|<L$.

Let us now study the stability of the closed-loop system (\ref{ConsensusDynamicA}). First, using the homogeneity theory, let us derive sufficient conditions for the design of control~\eqref{Eq:ProtocolA} such that the consensus is achieved in a fixed-time under switching topologies.

\begin{lemma}
\label{Lemma:Homogeneous}
The vector field of \eqref{ConsensusDynamicA} is homogeneous in the $0-$limit with degree $d_0=-1<0$ and homogeneous in the $+\infty$-limit with degree $d_\infty=qk-1>0$ with respect to the standard dilation.
\end{lemma}
\begin{proof}
It follows straightforwardly from the definition of homogeneity in the bi-limit \cite{Andrieu2008}.
\end{proof}

\begin{theorem}
\label{Th.ConsHom}
Let $\mathcal{F}=\{\mathcal{X}_1,\ldots,\mathcal{X}_m\}$ be a collection of connected graphs and let $\sigma(t):[0,\infty]\to \{1,\ldots,m\}$ be a non-Zeno switching signal.\\
If 
\begin{equation}
\kappa\zeta\geq L \text{ where } \kappa = \min_{l\in\{1,\ldots,n\}}{\kappa_l}
\end{equation} 
then, protocol~\eqref{Eq:ProtocolA} guarantees the fixed-time consensus on switched dynamic network $\mathcal{X}_{\sigma(t)}$ under arbitrary switching signals $\sigma(t)$.
\end{theorem}
\begin{proof}
Since, according to Lemma~\ref{Lemma:Homogeneous} the vector field in~\eqref{ConsensusDynamicA} is homogeneous in the bi-limit then, according to Theorem~\ref{Th:Fixed}, to show that fixed-time consensus is achieved it only remains to prove that~\eqref{ConsensusDynamicA} as well as
\begin{equation}
\label{Eq:HominZeroLimit}
\dot{x}=-\Phi_{0}(\mathcal{Q}(\mathcal{X}_{\sigma(t)})x)+\Delta(t)
\end{equation}
and
\begin{equation}
\label{Eq:HomInfLimit}
\dot{x}=-\Phi_{\infty}(\mathcal{Q}(\mathcal{X}_{\sigma(t)})x)
\end{equation}
are asymptotically stable where, with $d_0=-1$ and $d_\infty=qk-1$,
$$
-\Phi_{0}(\mathcal{Q}(\mathcal{X}_{\sigma(t)})x)+\Delta(t)=\lim_{\lambda \to 0} \lambda^{-(d_0+1)}(-\Phi(\mathcal{Q}(\mathcal{X}_{\sigma(t)})\lambda x)+\Delta(t))=\left[
\begin{array}{c}
\kappa_1 \zeta\sign{e_1} \\ \vdots \\ \kappa_n\zeta \sign{e_n}
\end{array}
\right]+\Delta(t)
$$
and 
$$
-\Phi_{\infty}(\mathcal{Q}(\mathcal{X}_{\sigma(t)})x)=\lim_{\lambda \to \infty} \lambda^{-(d_\infty+1)}(-\Phi(\mathcal{Q}(\mathcal{X}_{\sigma(t)})\lambda x)+\Delta(t))=\left[
\begin{array}{c}
\kappa_1\beta^k |e_1|^{qk}\sign{e_1} \\ \vdots \\ \kappa_n\beta^k |e_n|^{qk} \sign{e_n}
\end{array}
\right].
$$
To this aim, consider the (Lipschitz continuous) non-smooth Lyapunov function candidate
\begin{equation}
\label{Eq:LyapHomo}
    V(x)=\max(x_1,\cdots,x_n)-\min(x_1,\cdots,x_n),
\end{equation}
which is differentiable almost everywhere and positive definite. Let $x_i=\max(x_1,\cdots,x_n)$ and $x_j=\min(x_1,\cdots,x_n)$ for a nonzero interval $(\tau,\tau+\Delta_t)$ then the time derivative of~\eqref{Eq:LyapHomo} along the trajectory of~\eqref{ConsensusDynamicA} is given by
\begin{align*}
\label{Eq:DerLyapHomo}
    \mathcal{D}^{+}V(x)&=\kappa_i(\alpha|e_i|^{p} + \beta|e_i|^{q})^k\sign{e_i}-\kappa_j(\alpha|e_j|^{p} + \beta|e_j|^{q})^k\sign{e_j}+\kappa_i\zeta\sign{e_i}-\kappa_j\zeta\sign{e_j}+d_i(t)-d_j(t)\\
    &\leq\kappa_i(\alpha|e_i|^{p} + \beta|e_i|^{q})^k\sign{e_i}-\kappa_j(\alpha|e_j|^{p} + \beta|e_j|^{q})^k\sign{e_j}
\end{align*}
Notice that, since $x_i=\max(x_1,\cdots,x_n)$ and $x_j=\min(x_1,\cdots,x_n)$ then $e_i\leq0$ and $e_j\geq0$, and unless consensus is achieved, they cannot be both zero for a non zero interval, because since the graph is connected there is a path from agent $i$ to agent $j$ such that there is a node $i^*$ satisfying $x_{i}=x_{i^*}$ but $e_{i^*}<0$. Thus if $e_i(t)=0$ then $e_i(t^+)<0$. A similar argument follows to show that there is a node $j^*$ such that $x_{j}=x_{j^*}$ but $e_{j^*}>0$ and thus if $e_j(t)=0$ then $e_j(t^+)>0$. Thus $\mathcal{D}^{+}V(x)<0$ almost everywhere and consensus is achieved.


Fixed-time convergence for dynamic networks, under arbitrary switching, follows from the stability theory of switched systems~\cite[Theorem 2.1]{Liberzon2003} by noticing that~\eqref{Eq:LyapHomo} is a common Lyapunov function for each subsystem evolving with $\mathcal{X}_l$ connected.
\end{proof}

\begin{remark}
Notice that Theorem~\ref{Th.ConsHom} extends the results in~\cite{Gomez-Gutierrez2018}, in which the fixed-time consensus problem was addressed using homogeneity for the case where $k=1$ and no disturbances are present.
\end{remark}

Using an appropriate Lyapunov function, let us derive sufficient conditions for the design of control~\eqref{Eq:ProtocolA} such that the consensus is achieved in a predefined-time under a fixed topology.

\begin{theorem}\label{theorem_control1_fixed}
Let $\mathcal{X}_{l}$ be a connected graph and let 
\begin{equation}
\kappa_i\geq \frac{n\gamma(\rho)}{\lambda_2(\mathcal{Q}(\mathcal{X}_l))T_c}\text{ and } \kappa\zeta\geq L
\end{equation} 
where $$\kappa = \min_{l\in\{1,\ldots,n\}}{\kappa_l} $$ 
and $\gamma(\rho)$ is defined in Eq. (\ref{Eq:MinUpperEstimate}), then,  protocol~\eqref{Eq:ProtocolA} guarantees that consensus is achieved before a predefined-time $T_c$. 
\end{theorem}
\begin{proof}
Let $\delta = [\delta_1, \dots , \delta_n]^T$ be a disagreement variable such that $x = \alpha\mathbf{1} + \delta$ where $\alpha$ is a consensus value, which is unknown but constant. Consider the Lyapunov function candidate
\begin{equation}
V(\delta) = \frac{1}{n} \sqrt{\lambda_2(\mathcal{Q}(\mathcal{X}_l))\delta^T  \mathcal{Q}(\mathcal{X}_l)\delta},
\label{Eq:LyapunovProt1}
\end{equation}
By noticing that $\dot{\delta}=\dot{x}$, then it follows that
\begin{equation*}
\mathcal{D}^{+}V(\delta) = \frac{\sqrt{\lambda_2(\mathcal{Q}(\mathcal{X}_l))}}{n\sqrt{\delta^T\mathcal{Q}(\mathcal{X}_l)\delta}}\delta^T\mathcal{Q}(\mathcal{X}_l)\dot{\delta}  = \frac{\sqrt{\lambda_2(\mathcal{Q}(\mathcal{X}_l))}}{n\sqrt{\delta^T\mathcal{Q}(\mathcal{X}_l)\delta}}\delta^T\mathcal{Q}(\mathcal{X}_l)(-\Phi(\mathcal{Q}(\mathcal{X}_l)\delta)+\Delta(t)),
\end{equation*}
Let $v = \mathcal{Q}(\mathcal{X}_l)\delta = (v_1,\dots,v_n)^T$, therefore:
\begin{align}
\mathcal{D}^{+}V(\delta) =& \frac{\sqrt{\lambda_2(\mathcal{Q}(\mathcal{X}_l))}}{n}\left(-\frac{1}{\sqrt{\delta^T\mathcal{Q}(\mathcal{X}_l)\delta}} v^T\Phi(v) + \frac{\delta^T\mathcal{Q}(\mathcal{X}_l)}
{\sqrt{\delta^T\mathcal{Q}(\mathcal{X}_l)\delta}}\Delta(t)\right) \nonumber\\
  & = \frac{\sqrt{\lambda_2(\mathcal{Q}(\mathcal{X}_l))}}{n} \left(-\frac{1}{\sqrt{\delta^T\mathcal{Q}(\mathcal{X}_l)\delta}}\sum_{i=1}^n\kappa_i|v_i|(\alpha |v_i|^p + \beta|v_i|^q)^k - \frac{\zeta}{\sqrt{\delta^T\mathcal{Q}(\mathcal{X}_l)\delta}}\sum_{i=1}^n \kappa_i|v_i|  + \frac{\delta^T\mathcal{Q}(\mathcal{X}_l)}
{\sqrt{v^T\delta}}\Delta(t)\right)\label{prot2_firsteq}
\end{align}
Using Lemma \ref{Th:convex}, the first term can be rewritten as:
\begin{align*}
\frac{n}{\sqrt{\delta^T\mathcal{Q}(\mathcal{X}_l)\delta}}\sum_{i=1}^n\frac{1}{n}\kappa_i|v_i|(\alpha |v_i|^p + \beta|v_i|^q)^k &\geq \frac{ \kappa n}{\sqrt{\delta^T\mathcal{Q}(\mathcal{X}_l)\delta}}\sum_{i=1}^n\frac{1}{n}|v_i|(\alpha |v_i|^p + \beta|v_i|^q)^k\\ &\geq\frac{ \kappa n}{\sqrt{\delta^T\mathcal{Q}(\mathcal{X}_l)\delta}}\left(\frac{1}{n}\|v\|_1\right)\left(\alpha \left(\frac{1}{n}\|v\|_1\right)^p + \beta\left(\frac{1}{n}\|v\|_1\right)^q\right)^k
\end{align*}
where $\kappa = \min\{\kappa_1,\dots,\kappa_n\}$ and $\|v\|_1 = \sum_{i=1}^n |v_i|$
Moreover, using Lemma \ref{Th:Norm}, one can obtain
\begin{equation*}
\|v\|_1 \geq \|v\|_2 = v^Tv = \sqrt{\delta^T \mathcal{Q}(\mathcal{X}_l)^2\delta} 
\end{equation*}
Expressing the disagreement variable $\delta$ as a linear combination of the eigenvectors of $\mathcal{Q}(\mathcal{X}_l)$, 
the term $\delta^T \mathcal{Q}(\mathcal{X}_l)^2 \delta$ can be bounded as
\begin{equation*}
 \delta^T \mathcal{Q}(\mathcal{X}_l)^2 \delta 
 \geq \lambda_2(\mathcal{Q}(\mathcal{X}_l))\delta^T\mathcal{Q}(\mathcal{X}_l)\delta
\end{equation*}
Therefore, by Lemma \ref{Th:monotone}:
\begin{align}
\label{prot2_res1}
\frac{n}{\sqrt{\delta^T\mathcal{Q}(\mathcal{X}_l)\delta}}\sum_{i=1}^n\frac{1}{n}\kappa_i|v_i|(\alpha |v_i|^p + \beta|v_i|^q)^k &\geq \frac{\kappa n}{\sqrt{\delta^T\mathcal{Q}(\mathcal{X}_l)\delta}}V(\alpha V^p + \beta V^q)^k = \kappa\sqrt{\lambda_2(\mathcal{Q}(\mathcal{X}_l))} (\alpha V^p + \beta V^q)^k
\end{align}
Furthermore, from the last two terms of \eqref{prot2_firsteq} the following is obtained:
\begin{align}
- \frac{\zeta}{\sqrt{\delta^T\mathcal{Q}(\mathcal{X}_l)\delta}}\sum_{i=1}^n \kappa_i|v_i|  + \frac{v^T}
{\sqrt{\delta^T\mathcal{Q}(\mathcal{X}_l)\delta}}\Delta(t) &\leq \lambda_2(\mathcal{Q}(\mathcal{X}_l))\left(-\frac{\zeta}{\|v\|}\sum_{i=1}^n \kappa_i|v_i|  + \frac{v^T}
{\|v\|}\Delta(t)\right)\nonumber \\
& \leq\lambda_2(\mathcal{Q}(\mathcal{X}_l))(-\zeta\kappa + L) \leq 0 
\label{prot2_res2}
\end{align}
Therefore, the following inequality is obtained from \eqref{prot2_firsteq}, by combining \eqref{prot2_res1} and \eqref{prot2_res2}:
\begin{equation}
\mathcal{D}^{+}V(\delta)  \leq -\frac{\kappa\lambda_2(\mathcal{Q}(\mathcal{X}_l))}{n}(\alpha V^p + \beta V^q)^k \leq -\frac{\gamma(\rho)}{T_c}(\alpha V^p + \beta V^q)^k
\end{equation}
Then, according to Theorem \ref{thm:weak_pt}, protocol~\eqref{Eq:ProtocolA} guarantees that consensus is achieved before a predefined-time $T_c$.
\end{proof}

\begin{example}
\label{Example_fdd_static}
Consider the multi-agent system (\ref{Eq:AgentDyn}) composed of $n=10$ agents with external perturbation $d_i(t) = \sin(40t + 0.1i)$. The communication topology, given in Figure~\ref{Fig:net0} is undirected and static. The corresponding algebraic connectivity is $\lambda_2(\mathcal{Q}(\mathcal{X}))=0.27935$. The initial conditions of the agents are randomly generated and are as follows:
$$x(0)=[134.51,   40.72 , 214.15,   40.04, -241.50, -189.57,  181.35,   -7.8517,  172.42, -145.29]^T.$$ 
According to Theorem \ref{theorem_control1_fixed}, protocol~\eqref{Eq:ProtocolA} with $T_c=1$, $p = 1.5, q = 3.0, k=0.5, \alpha = 1, \beta = 2, \kappa = 178.88, \zeta = 0.0177$ guarantees that the consensus is achieved before a predefined-time $T_c$ under the graph topology $\mathcal{X}$. Figure~\ref{Fig:FDD_static} shows the corresponding result. For this experiment the settling time is of 0.095s. 

\begin{figure}[t]
\begin{center}         
\includegraphics[width=6cm]{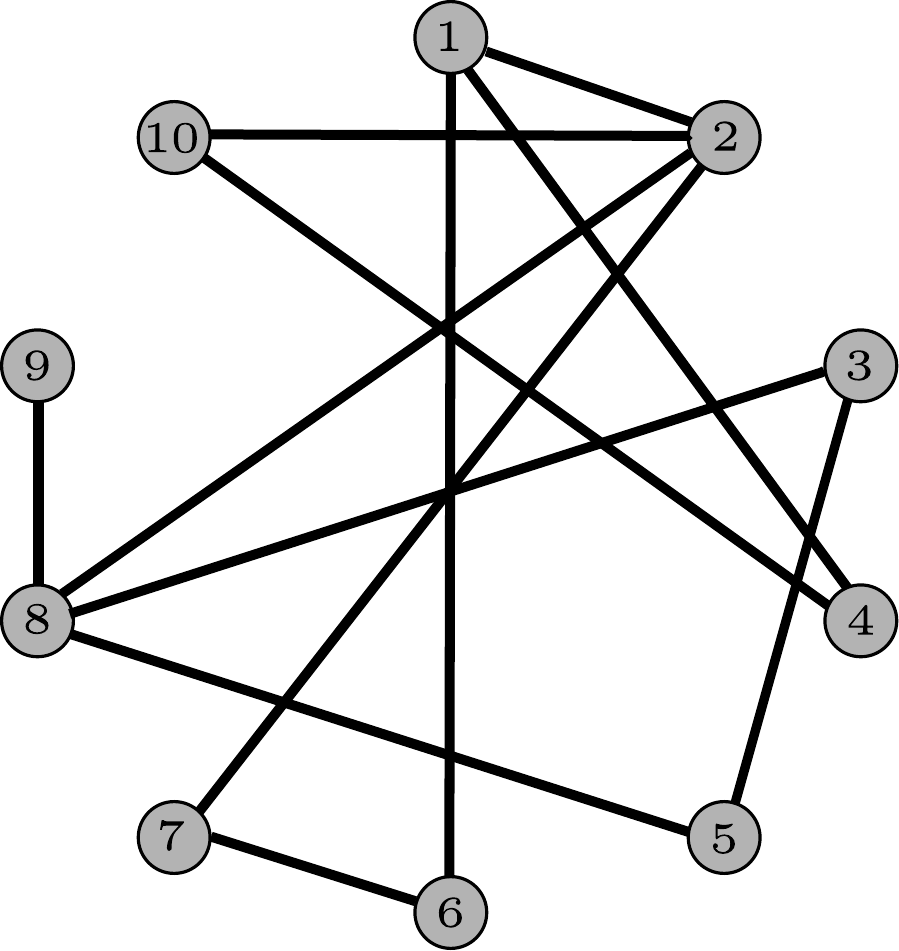}
\end{center}
\caption{Static network $\mathcal{X}$ used for Example \ref{Example_fdd_static}.}\label{Fig:net0}
\end{figure}

\graphicspath{{Simulations/}}
\begin{figure}[t]
\begin{center}
 \def\svgwidth{15cm}
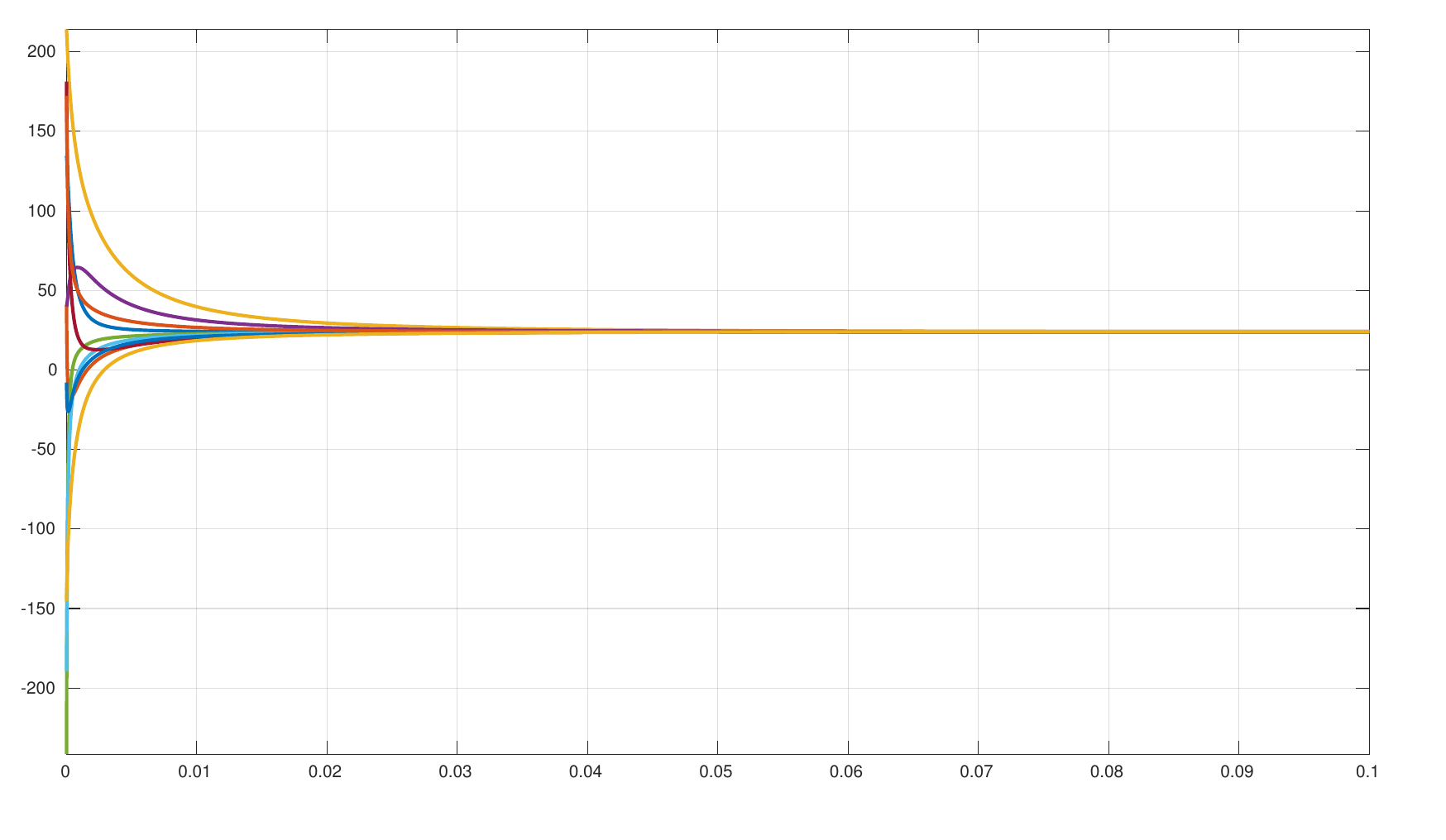
\end{center}
\caption{ Convergence of the consensus algorithm for Example \ref{Example_fdd_static}.}
\label{Fig:FDD_static}
\end{figure}
\end{example}

\begin{example}
\label{Example_fdd}
Consider the multi-agent system (\ref{Eq:AgentDyn}) composed of $n=10$ agents with external perturbation $d_i(t) = \sin(40t + 0.1i)$. The collection of communication topologies $\mathcal{F}=\{\mathcal{X}_1,\ldots,\mathcal{X}_4\}$, given in Figure~\ref{fig:net1}-\ref{fig:net4} are undirected. The switched dynamic network evolves according to the switching signal $\sigma(t)$ given in Figure~\ref{Fig:FDD_plot} which satisfies the minimum dwell time condition. The corresponding algebraic connectivity is $$\lambda_2(\mathcal{Q}(\mathcal{X}_l)) \in \{0.16548 ,  0.73648 ,  0.15776 ,  0.57104\}.$$ The initial conditions of the agents are randomly generated and are as follows:
$$x(0)=[-210.02,  117.66,  161.32,  -78.30, -181.93,   82.97,  165.22,   86.81, -180.27,  -60.58
]^T.$$
According to Theorem \ref{Th.ConsHom}, protocol~\eqref{Eq:ProtocolA} with $p = 1.5$, $q = 3.0$, $k=0.5$, $\alpha = 1$, $\beta = 2$, $$\{\kappa_1,\dots,\kappa_4\}=\{301.9585,   67.8472,  316.7348,   87.5037\},$$ $\kappa = 67.8472$ and $\zeta=0.0466$ guarantees that the consensus is achieved in a fixed-time under the switched dynamic network $\mathcal{X}_{\sigma(t)}$. Figure~\ref{Fig:FDD_plot} shows the corresponding result. 

\begin{figure}
\begin{center}  
\begin{subfigure}[b]{.3\linewidth}
\includegraphics[width=\linewidth]{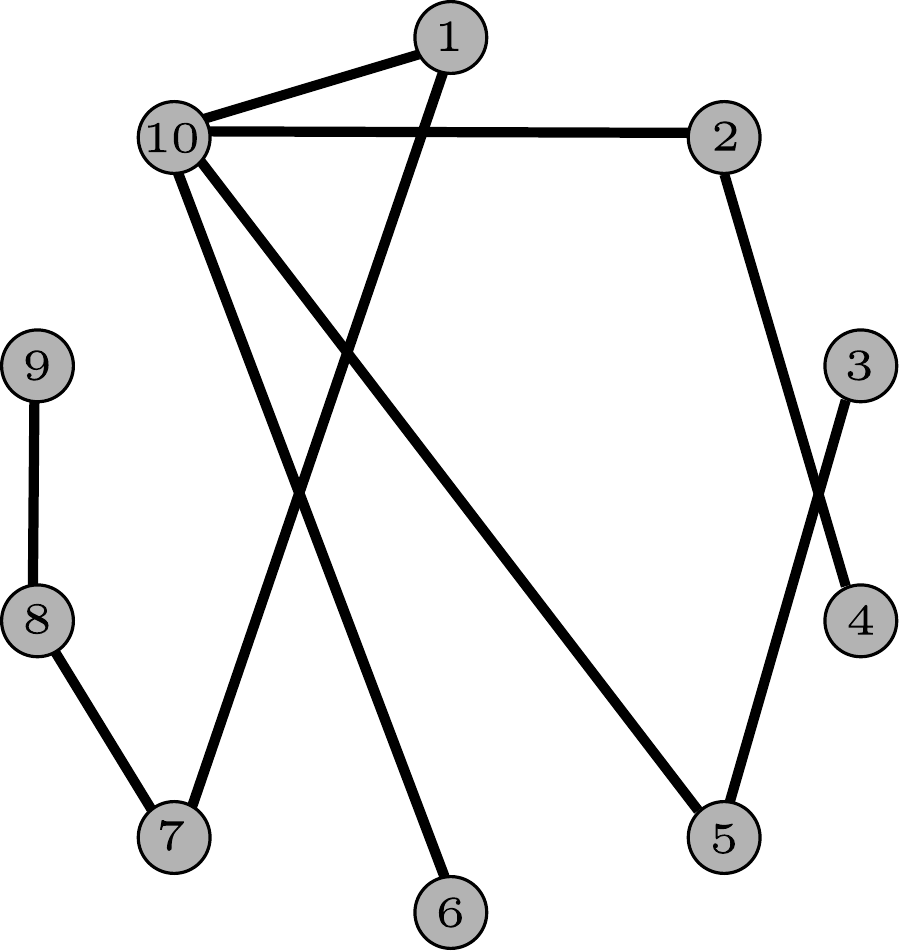}
\caption{$\mathcal{X}_1$}\label{fig:net1}
\end{subfigure}
\hspace{0.02\textwidth}
\begin{subfigure}[b]{.3\linewidth}
\includegraphics[width=\linewidth]{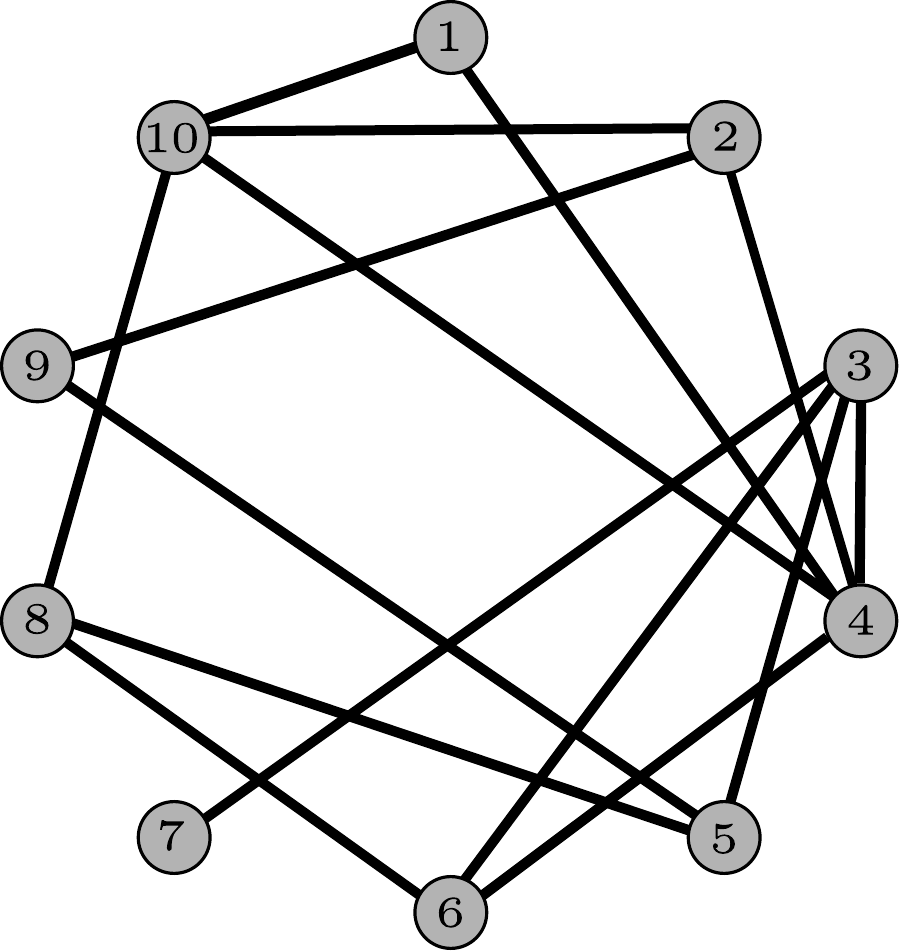}
\caption{$\mathcal{X}_2$}\label{fig:net2}
\end{subfigure}
\end{center}

\begin{center} 
\begin{subfigure}[b]{.3\linewidth}
\includegraphics[width=\linewidth]{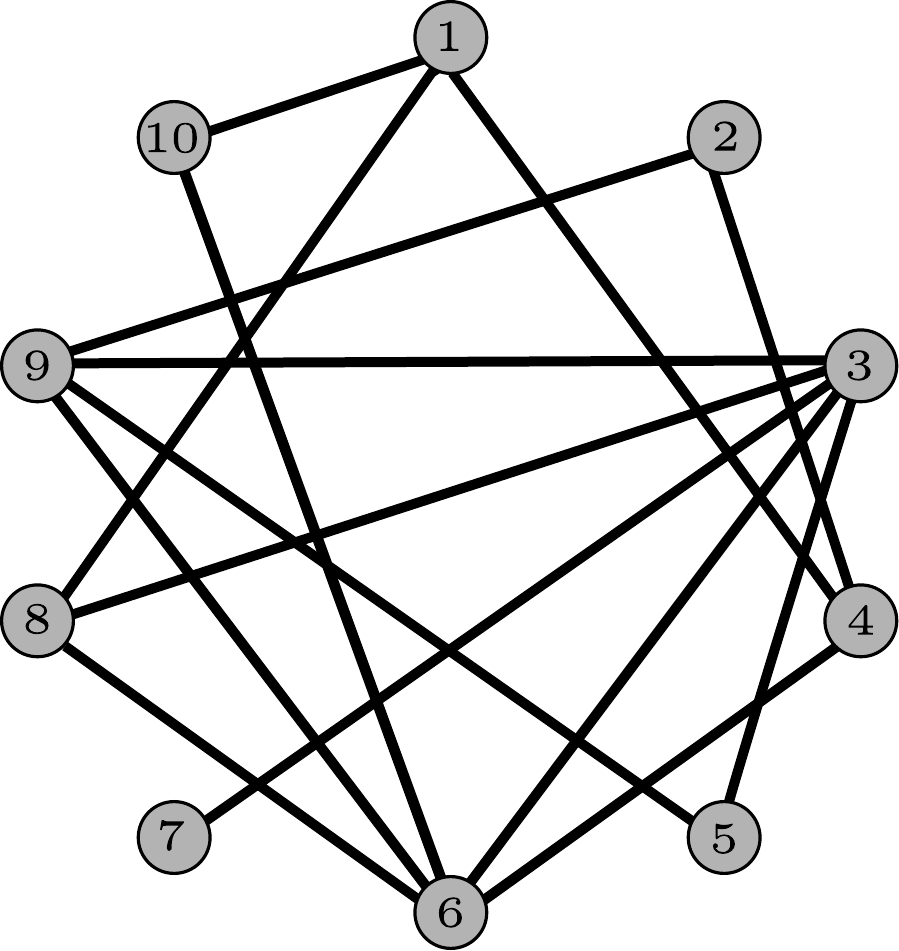}
\caption{$\mathcal{X}_3$}\label{fig:net3}
\end{subfigure}
\hspace{0.02\textwidth}
\begin{subfigure}[b]{.3\linewidth}
\includegraphics[width=\linewidth]{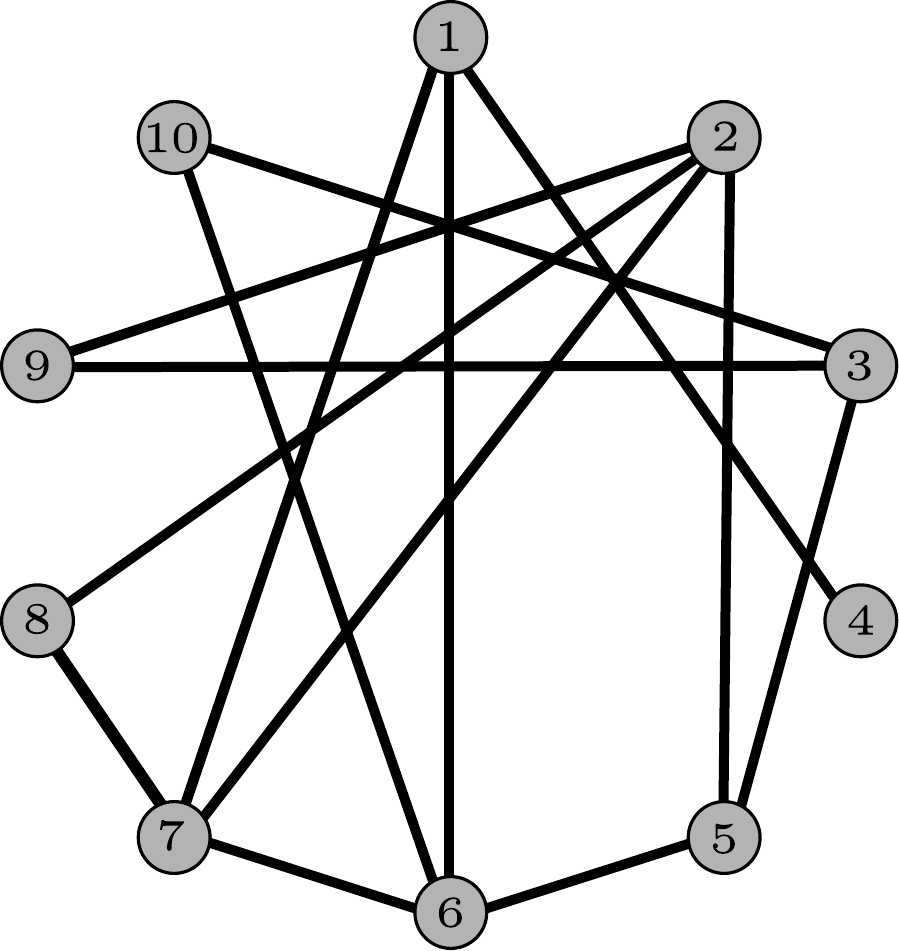}
\caption{$\mathcal{X}_4$}\label{fig:net4}
\end{subfigure}
\end{center}
\caption{Collection of communication topologies $\mathcal{F}=\{\mathcal{X}_1,\ldots,\mathcal{X}_4\}$ used for the switched dynamic network in Example \ref{Example_fdd}.}
\end{figure}

\graphicspath{{Simulations/}}
\begin{figure}[t]
\begin{center}
 \def\svgwidth{15cm}
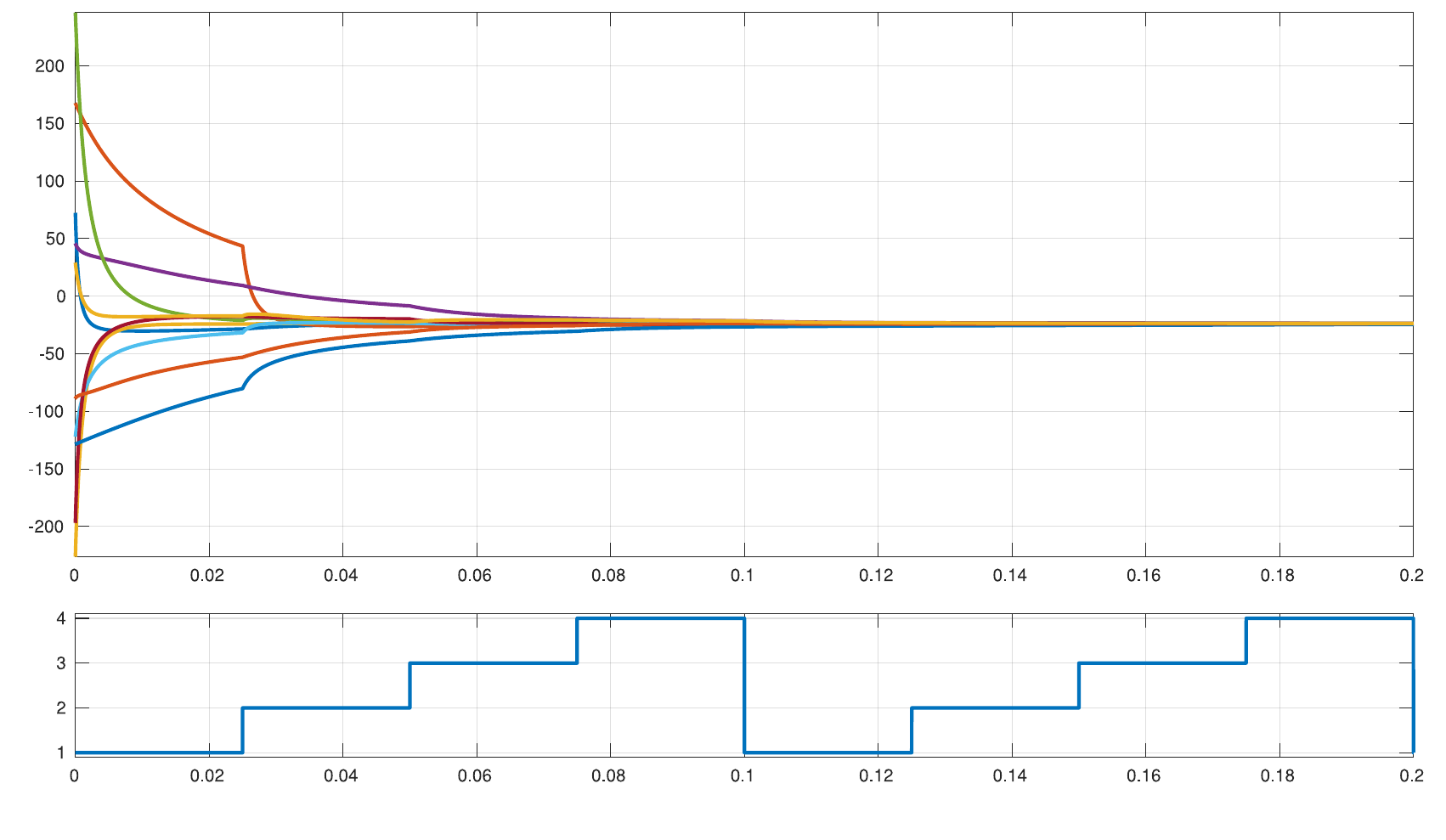
\end{center}
\caption{Convergence of the consensus algorithm for Example \ref{Example_fdd} and the corresponding switching signal.}
\label{Fig:FDD_plot}
\end{figure}

\end{example}

\subsubsection{Convergence analysis under consensus protocol (\ref{Eq:ProtocolB})}

Similarly to the previous subsection, let us define $x=[x_1 \ \cdots \ x_n]^T$. Notice that the dynamic of the network under the consensus algorithm \eqref{Eq:ProtocolB} is given by
\begin{equation}
\label{ConsensusDynamicB}
\dot{x}=-D(\mathcal{X}_{\sigma(t)})\Phi(D(\mathcal{X}_{\sigma(t)})^Tx)+\Delta(t).
\end{equation}
where, for $z=[z_1 \ \cdots \ z_n]^T\in\mathbb{R}^n$, the function $\Phi:\mathbb{R}^n\rightarrow \mathbb{R}^n$ is defined as \eqref{Eq:Fp}
and $\Delta(t)=[d_1(t) \ \cdots \ d_n(t)]^T$ with $\|\Delta(t)\|<L$.

Let us now study the stability of the closed-loop system (\ref{ConsensusDynamicB}). First, let us derive a useful lemma concerning average consensus in the absence of disturbance. Then, using an appropriate Lyapunov function, let us derive sufficient conditions for the design of control \eqref{Eq:ProtocolB} such that the consensus is achieved in a predefined-time under switching topologies

\begin{lemma}
\label{DeltaOrth}
Let $\mathcal{X}_{\sigma(t)}$ be a switching dynamic network composed of connected graphs and assume that under the protocol~\eqref{Eq:ProtocolB} and in the absence of disturbance consensus is achieved at a time $T_c$. Then, the consensus state is $x^*=\frac{1}{n}\mathbf{1}^Tx(t_0)$, i.e. consensus to the average is achieved.
\end{lemma}
\begin{proof}
Let $s_x=\mathbf{1}^Tx$ be the sum of the agent states. In the absence of disturbances, i.e. if $\Delta(t)=0$, since $\mathbf{1}^TD(\mathcal{X}_{\sigma(t)})=0$, then $\dot{s}_x=-\mathbf{1}^T D(\mathcal{X}_{\sigma(t)})\Phi(D(\mathcal{X}_{\sigma(t)})^Tx)=0$. Thus, $s_x$ is constant during the evolution of the system, i.e. $\forall t\geq 0$, $s_x(t)=\mathbf{1}^Tx(t_0)=s_x(t_0)$. 


Therefore, 
$
s_x(t)-s_x(t_0)=\mathbf{1}^Tx-\mathbf{1}^Tx(t_0)=0.
$
Thus, if $t\geq T_c$, $x_1=\cdots=x_n=x^*$. Thus, $\mathbf{1}^Tx=nx^*$ and $x^*=\frac{1}{n}\mathbf{1}^Tx(t_0)$, i.e. consensus to the average is achieved.
\end{proof}

\begin{theorem}\label{eq_theorem2}
Let $\mathcal{F}=\{\mathcal{X}_1,\ldots,\mathcal{X}_m\}$ be a collection of connected graphs and let $\sigma(t):[0,\infty]\to \{1,\ldots,m\}$ be a non-Zeno switching signal. If 
\begin{equation}
\kappa_i\geq \frac{M\gamma(\rho)}{\lambda_2^*T_c}\text{ and } \zeta\geq\frac{L}{\kappa\sqrt{\lambda_2^*}}
\end{equation} 
with $$M = \min_{l\in\{1,\ldots,m\}}{|\mathcal{E}(\mathcal{X}_{l})|}, \ \kappa = \max_{l\in\{1,\ldots,n\}}{\kappa_l} \text{ and } \lambda_2^* = \min_{l\in\{1,\ldots,m\}}{\lambda_2(\mathcal{Q}(\mathcal{X}_{l}))},$$ 
and $\gamma(\rho)$ is defined in Eq. (\ref{Eq:MinUpperEstimate}), then, protocol~\eqref{Eq:ProtocolB} guarantees that consensus is achieved before a predefined-time $T_c$ on switched dynamic networks $\mathcal{X}_{\sigma(t)}$ under arbitrary switching signals $\sigma(t)$. Moreover, in the absence of disturbances consensus to the average is obtained.
\end{theorem}
\begin{proof}
Let $\delta = [\delta_1, \dots , \delta_n]^T$ be a disagreement variable such that $x = \alpha\mathbf{1} + \delta$ where $\alpha$ is a consensus value, which is unknown but constant. 
Consider the Lyapunov function candidate
\begin{equation}
V = \frac{1}{M} \sqrt{\lambda_2^*\delta^T  \delta}.
\label{Eq:LyapunovProt1_0}
\end{equation}
To show that consensus is achieved on dynamic networks under arbitrary switchings, we will prove that~\eqref{Eq:LyapunovProt1_0} is a common Lyapunov function for each subsystem of the switched nonlinear system~\eqref{ConsensusDynamicB}~\cite[Theorem 2.1]{Liberzon2003}. To this aim, assume that $\sigma(t)=l$ for $t\in[0,T_c]$. By noticing that $\dot{\delta}=\dot{x}$, then it follows that
\begin{equation*}
\mathcal{D}^{+}V(x) = \frac{\sqrt{\lambda_2^*}}{M\sqrt{\delta^T\delta}}\delta^T\dot{\delta}  = -\frac{\sqrt{\lambda_2^*}}{M\sqrt{\delta^T\delta}}\delta^T(D(\mathcal{X}_l)\Phi(D(\mathcal{X}_l)^T\delta)+\Delta(t)),
\end{equation*}
Let $v = D(\mathcal{X}_l)^T\delta = (v_1,\dots,v_m)^T$, therefore:
\begin{align}
\mathcal{D}^{+}V(x) =& \frac{\sqrt{\lambda_2^*}}{M}\left(-\frac{1}{\|\delta\|} v^T\Phi_p(v) + \frac{\delta^T}
{\|\delta\|}\Delta(t)\right)\nonumber \\
  & = \frac{\sqrt{\lambda_2^*}}{M} \left(-\frac{1}{\|\delta\|}\sum_{i=1}^m\kappa_i|v_i|(\alpha |v_i|^p + \beta|v_i|^q)^k - \frac{\zeta}{\|\delta\|}\sum_{i=1}^m \kappa_i|v_i|  + \frac{\delta^T}
{\|\delta\|}\Delta(t)\right).\label{prot1_firsteq}
\end{align}
Using Lemma \ref{Th:convex}, the first term can be rewritten as:
\begin{align*}
\frac{m}{\|\delta\|}\sum_{i=1}^m\frac{1}{m}\kappa_i|v_i|(\alpha |v_i|^p + \beta|v_i|^q)^k &\geq \frac{ \kappa m}{\|\delta\|}\sum_{i=1}^m\frac{1}{m}|v_i|(\alpha |v_i|^p + \beta|v_i|^q)^k\\ &\geq\frac{ \kappa m}{\|\delta\|}\left(\frac{1}{m}\|v\|_1\right)\left(\alpha \left(\frac{1}{m}\|v\|_1\right)^p + \beta\left(\frac{1}{m}\|v\|_1\right)^q\right)^k,
\end{align*}
where $\kappa = \min\{\kappa_1,\dots,\kappa_n\}$ and $\|v\|_1 = \sum_{i=1}^m |v_i|$.
Moreover, it follows from Lemma~\ref{Th:Norm} that
\begin{equation*}
\|v\|_1 \geq \|v\|_2 = v^Tv = \sqrt{\delta^T \mathcal{Q}(\mathcal{X}_l)\delta} \geq \sqrt{\lambda_2^*}\|\delta\|.
\end{equation*}
Therefore, by Lemma \ref{Th:monotone}:
\begin{align}
\label{prot1_res1}
\frac{ m}{\|\delta\|}\sum_{i=1}^m\frac{1}{m}\kappa_i|v_i|(\alpha |v_i|^p + \beta|v_i|^q)^k &\geq \frac{\kappa m}{\|\delta\|}V(\alpha V^p + \beta V^q)^k \geq \kappa\sqrt{\lambda_2^*} (\alpha V^p + \beta V^q)^k.
\end{align}
Furthermore, from the last two terms of \eqref{prot1_firsteq}, the following is obtained:
\begin{equation}
\label{prot1_res2}
- \frac{\zeta}{\|\delta\|}\sum_{i=1}^m \kappa_i|v_i|  + \frac{\delta^T}
{\|\delta\|}\Delta(t) \leq -\frac{\kappa \zeta}{\|\delta\|}\|v\|_1 + \|\Delta(t)\|\leq -\zeta\kappa\sqrt{\lambda_2^*} + L \leq 0
\end{equation}
Therefore, the following inequality is obtained from \eqref{prot1_firsteq}, by combining \eqref{prot1_res1} and \eqref{prot1_res2}:
\begin{equation}
\mathcal{D}^{+}V(x) \leq -\frac{\kappa\lambda_2^*}{M}(\alpha V^p + \beta V^q)^k \leq -\frac{\gamma(\rho)}{T_c}(\alpha V^p + \beta V^q)^k
\end{equation}
Then, according to Theorem~\ref{thm:weak_pt}, protocol~\eqref{Eq:ProtocolB} guarantees that the consensus is achieved before a predefined-time $T_c$. Moreover, since the above argument holds for any $\mathcal{X}_l\in\mathcal{F}$, then protocol~\eqref{Eq:ProtocolB} guarantees that the consensus is achieved before a predefined-time $T_c$, on switching dynamic networks under arbitrary switching. Furthermore, it follows from Lemma~\ref{DeltaOrth} that if $\Delta(t)=0$ then the consensus state $x^*$ is the average of the initial values of the agents. 
\end{proof}

\begin{remark}
Note that existing consensus protocols such as~\cite{Zuo2014,Zuo2014a,Ning2017b,Wang2017b} are subsumed in our approach. In fact, the consensus protocols in~\cite{Zuo2014,Ning2017b,Wang2017b} are derived from~\eqref{Eq:ProtocolA} and \eqref{Eq:ProtocolB} by taking $k=1$, $p=1-s$ and $q=1+s$, $s\in(0,1)$. Moreover, we show that our analysis provides less conservative estimate than the ones proposed, for instance in~\cite{Zuo2014a}. Thus, our paper contributes to a broader class of consensus protocols with fixed-time convergence. Moreover, unlike~\cite{Zuo2014,Zuo2014a,Ning2017b}, in our method the convergence time is determined a priori.
\end{remark}

\begin{remark}
Notice that, the Lyapunov functions in our proofs are different than in previous methods, such as \cite{Wang2010,Zuo2014,Zuo2014a,Ning2017b}. Consequently, the proposed predefined-time consensus results do not follow trivially from the existing literature. An essential part of our approach is the convexity result provided in Lemma~\ref{Th:convex_fundamental}.
\end{remark}

\begin{example}
\label{Example_dfd}
Consider the multi-agent system (\ref{Eq:AgentDyn}) composed of $n=10$ agents with external perturbation $d_i(t) = \sin(40t + 0.1i)$. The collection of communication topologies, given in Figure~\ref{fig:net1}-\ref{fig:net4} are undirected. The switched dynamic network evolves according to the switching signal $\sigma(t)$ given in Figure~\ref{Fig:DFD_plot} which satisfies the minimum dwell time condition. The corresponding algebraic connectivity is $\lambda_2(\mathcal{Q}(\mathcal{X}_l)) \in \{0.16548 ,  0.73648 ,  0.15776 ,  0.57104\}$. The initial conditions of the agents are randomly generated and are as follows:
$$x(t_0)=[72.31  167.49 -226.30,   45.68,  246.20, -121.78, -196.90, -128.59,  -88.57,   29.05]^T.$$
According to Theorem \ref{eq_theorem2}, protocol~\eqref{Eq:ProtocolB} with $T_c=1$, $p = 1.5, q = 3.0, k=0.5, \alpha = 1, \beta = 2, \{\kappa_1,\dots,\kappa_4\}=\{241.5668,  54.2777,  253.3879,   70.0029\},\kappa = 54.2777$ and $\zeta = 0.3693$ guarantees that the consensus is achieved in a predefined-time $T_c$ under the switched dynamic network $\mathcal{X}_{\sigma(t)}$. Figure~\ref{Fig:FDD_plot} shows the corresponding result. For this example, the settling time is 0.187s.

\begin{figure}[t]
\begin{center}
\graphicspath{{Simulations/}}
 \def\svgwidth{15cm}
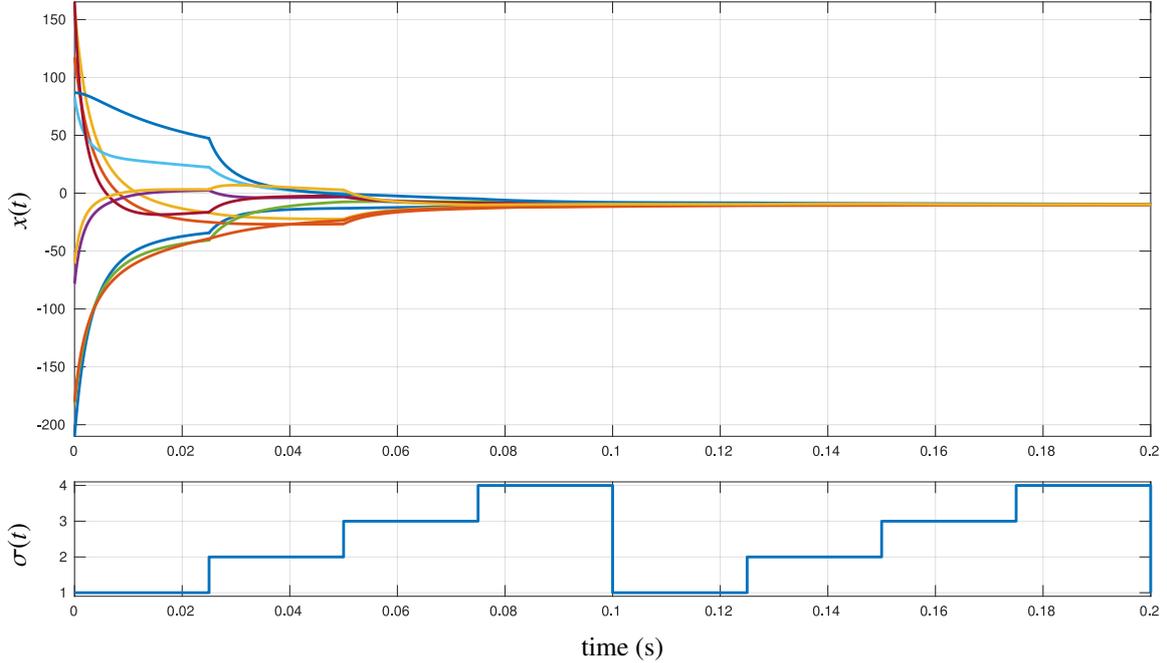
\end{center}
\caption{Convergence of the consensus algorithm for Example \ref{Example_dfd} and the corresponding switching signal.}\label{Fig:DFD_plot}
\end{figure}

\end{example}

\section{Comparison and Discussion}
\label{Sec:Disc}

In this section, we present a discussion on the contribution with respect to the state of the art in the literature. Our main argument is that our approach subsumes existing methods, as it allows a broader range of parameter selection, which provides extra degrees of freedom to design more efficient consensus protocols (for instance, allowing to reduce the slack between the true convergence time and the predefined-convergence time). 

Table~\ref{Tab:ComparisonA} and Table~\ref{Tab:ComparisonB} present a comparison between the current proposal and some of the existing fixed-time consensus protocols for the leaderless consensus problem when using autonomous protocols. Notice that contrary to the results existing in the literature, our approach does not present any restriction on the parameters $p$, $k$ and $q$ (additional to those given in Theorem~\ref{th:tf_poly} for~\eqref{Eq:HomFixedPoly}). Regarding consensus protocols in the form~\eqref{Eq:ProtocolA}, it is illustrated in Table~\ref{Tab:ComparisonA} that for switched dynamic networks under arbitrary switching among connected graphs, our approach provides fixed-time convergence (no upper bound convergence is provided in this case); whereas for static networks, it provides predefined-time convergence, even if the agents are affected by disturbances. Regarding consensus protocols in the form~\eqref{Eq:ProtocolB}, see~Table~\ref{Tab:ComparisonB}, for switched dynamic networks among connected graphs, our approach provides predefined-time convergence for agents with single integrator dynamics affected by disturbances. 

Regarding the estimation of the upper bound of the convergence time, the results in~\cite{Zuo2014a} are based on the estimate provided in~\cite{Polyakov2012}, which has been shown in~\cite{Aldana-Lopez2018} to be very conservative. Moreover, to address the robust consensus problem in~\cite{Zuo2016} an estimate improving the result in~\cite{Polyakov2012} was presented for the case where $k=1$, $q=\frac{\hat{m}}{\hat{n}}$, $p=\frac{\hat{p}}{\hat{q}}$ satisfying
$$
0<\frac{\hat{q}(\hat{m}-\hat{n})}{\hat{n}(\hat{q}-\hat{p})}\leq 1.
$$
However, such estimation is also too conservative. The proposed approach is based on the predefined-time stability results provided in~\cite{Aldana-Lopez2018} which extend the results in~\cite{Parsegov2012} (in~\cite{Parsegov2012}, they were limited to $k=1$, $p=1-s$, $q=1+s$, $0< s <1$) by showing that~\eqref{Eq:HomFixedPoly} is predefined-time stable, with $T_c$ as the least upper bound of the convergence-time, for every $p,q,k>0$ satisfying $kp<1$ and $kq>1$. Thus, our approach has the same upper estimate as~\cite{Ning2017b} for the case when $k=1$, $p=1-s$, $q=1+s$, $0< s <1$ but provides a broader class of consensus protocols than \cite{Gomez-Gutierrez2018,Zuo2014,Zuo2014a,Ning2017b,Wang2017b} by having extra degrees of freedom on $p$, $q$, $k$ as long as the conditions $p,q,k>0$, $kp<1$ and $kq>1$ are met. As we show below, such broader parameter selection allows to design consensus protocols with improved estimates of the convergence time bound.

Regarding consensus protocols in the form~\eqref{Eq:ProtocolA}, in addition to the above mentioned advantages, in this paper we show, using an homogeneity approach, that~\eqref{Eq:ProtocolA} is a fixed-time consensus algorithm for switched dynamic networks under arbitrarily switching among connected graphs, which extends the algorithm in~\cite{Gomez-Gutierrez2018} to allow $k\neq1$ and disturbances affecting the agents.

\begin{table}
    \centering
    \begin{tabular}{|c|p{4cm}|p{1.6cm}|p{1.6cm}|p{3.5cm}|}
    \hline
        Reference & Additional restrictions on $k$, $p$, $q$& Network Type & Robustness & Upper bound Estimate/ Comparison to ours\\
        \hline         
        \cite{Gomez-Gutierrez2018} & $k=1$& Jointly Connected & No &  None provided. Based on homogeneity. \\
         \hline
         \cite{Zuo2014} & $k=1$, $p=1-s$, $q=1+s$, $0< s <1$ & Static Connected &  No & from~\cite{Parsegov2012}, Same \\
          \hline
         \cite{Zuo2014a} & $k=1$ & Static Connected & No & from~\cite{Polyakov2012}, More conservative \\
         \hline
         \cite{Zuo2016} & $k=1$ & Static Connected & Yes & from~\cite{Polyakov2012}, More Conservative \\
         \hline
         \cite{Zuo2016} & $k=1$, $q=\frac{\hat{m}}{\hat{n}}$, $p=\frac{\hat{p}}{\hat{q}}$ satisfying
$0<\frac{\hat{q}(\hat{m}-\hat{n})}{\hat{n}(\hat{q}-\hat{p})}\leq 1.
$ & Static Connected & Yes & from~\cite{Zuo2016}, More Conservative \\
         \hline
         \cite{Ning2017b} & $k=1$, $p=1-s$, $q=1+s$, $0< s <1$ & Static Connected &  Yes  &  from~\cite{Parsegov2012}, Same\\
         \hline
         Ours & None & Dynamic Connected &  Yes  & None provided. Based on homogeneity. \\
         \hline
         Ours & None & Static Connected &  Yes  & from~\cite{Aldana-Lopez2018} \\
         \hline
     \end{tabular}
     \caption{Comparison of papers presenting autonomous protocols of the form~\eqref{Eq:ProtocolA} for the consensus problem.}
     \label{Tab:ComparisonA}
 \end{table}

 \begin{table}
    \centering
    \begin{tabular}{|c|p{4cm}|p{1.6cm}|p{2cm}|p{2.6cm}|}
    \hline
        Reference & Additional restrictions on $k$, $p$, $q$& Network Type &  Robustness & Upper bound Estimate/ Comparison to ours\\
        \hline
        \cite{Parsegov2013} & $k=1$ & Static Connected & No & from~\cite{Polyakov2012}, More Conservative \\
        \hline
        \cite{Parsegov2013} & $k=1$, $p=1-s$, $q=1+s$, $0< s <1$ & Static Connected & No & from~\cite{Parsegov2012}, Same \\
        \hline
        \cite{Zuo2014} & $k=1$, $p=1-s$, $q=1+s$, $0< s <1$ & Static Connected & No & from~\cite{Parsegov2012}, Same \\
         \hline
        \cite{Zuo2014a} & $k=1$ & Dynamic Connected & No & from~\cite{Polyakov2012}, More Conservative \\
        \hline
        \cite{Zuo2016} & $k=1$ & Static Connected & Yes & from~\cite{Polyakov2012}, More Conservative \\
        \hline
        \cite{Zuo2016} & $k=1$, $q=\frac{\hat{m}}{\hat{n}}$, $p=\frac{\hat{p}}{\hat{q}}$ satisfying
$0<\frac{\hat{q}(\hat{m}-\hat{n})}{\hat{n}(\hat{q}-\hat{p})}\leq 1.
$ & Static Connected & Yes & from~\cite{Zuo2016}, More Conservative \\
        \hline
        \cite{Hong2017} & $k=1$ & Static Connected & Yes & from~\cite{Polyakov2012}, More Conservative \\
        \hline
        \cite{Wang2017b} & $k=1$, $p=1-s$, $q=1+s$, $0< s <1$ & Dynamic Connected & No & from~\cite{Parsegov2012}, Same\\
        \hline
        \cite{Ni2017} & $k=1$, $p=1-s$, $q=1+s$, $0< s <1$ & Dynamic Connected & No & from~\cite{Parsegov2012}, Same\\
        \hline
        \cite{Ning2017b} & $k=1$, $p=1-s$, $q=1+s$, $0< s <1$ & Dynamic Connected &  Yes &  from~\cite{Parsegov2012}, Same\\
        \hline
        Ours & None & Dynamic Connected & Yes & from~\cite{Aldana-Lopez2018} \\
        \hline
    \end{tabular}
    \caption{Comparison of papers presenting autonomous protocols of the form~\eqref{Eq:ProtocolB} for the consensus problem.}
    \label{Tab:ComparisonB}
\end{table}

\subsection{Comparison}

Notice from Table~\ref{Tab:ComparisonA} and Table~\ref{Tab:ComparisonB} that previous fixed-time consensus protocols were limited by a parameter selection with $k=1$ and most of them allow only $p=1-s$, $q=1+s$, $0< s <1$. In this subsection, we present numerical simulations and comparisons to illustrate that by having greater flexibility in the protocol design, one can obtain protocols with improved estimates of the convergence time bound.

In our opinion, relevant robust consensus algorithms for agents with single integrator dynamics are given in~\cite{Ning2017b} for the leaderless consensus problem. We show that, by selecting parameters different from the ones proposed in~\cite{Ning2017b}, we obtain robust consensus protocols where the slack between the real convergence time and the estimated upper bound is lower. 

This simulation results are presented in Example~\ref{Example_comp_fdd} for protocol~\eqref{Eq:ProtocolA} and in Example~\ref{Example_comp_dfd} for protocol~\eqref{Eq:ProtocolB}. The first case, $p = 0.1, q = 0.9, k=1$ corresponds to a selection allowed by~\cite{Ning2017b} for which the gains are computed from~\cite{Ning2017b} to have an upper bound for the convergence time $T_c=1s$. For the second case, our aim is to show that by varying $k\neq 1$, the slack between the true convergence time and the upper bound of the convergence time $T_c=1$ can be reduced. Moreover, we show that such slack can be further reduced by selecting values for $p$ and $q$ that are not symmetrical w.r.t. one, i.e. values that are outside the parameter selection allowed in~\cite{Ning2017b}. Thus, our approach subsumes existing fixed-time consensus protocols and provides more flexibility in the parameter selection to improve, for instance, the convergence. 

\begin{example}
\label{Example_comp_fdd}
Consider the multi-agent system (\ref{Eq:AgentDyn}) composed of $n=10$ agents with external perturbation $d_i(t) = \sin(40t + 0.1i)$.. The collection of communication topologies, given in Figure~\ref{fig:net1}-\ref{fig:net4} are undirected. The switched dynamic network evolves according to the switching signal $\sigma(t)$ given in Figure~\ref{Fig:fdd_comp} which satisfies the minimum dwell time condition. The initial conditions of the agents are randomly generated and are as follows:
$$x(t_0)=[-210.02,  117.66,  161.32,  -78.30, -181.93,   82.97,  165.22,   86.81, -180.27,  -60.58
]^T.$$
Protocol~\eqref{Eq:ProtocolA} with, $\alpha = 1, \beta = 2, \{\kappa_1,\dots,\kappa_4\}=\{301.9585,   67.8472,  316.7348,   87.5037\}, \kappa = 67.8472$ and $\zeta=0.0466$ and the parameter selection described in Table~\ref{Tab:Example_fdd} achieves consensus in fixed-time. Figure~\ref{Fig:fdd_comp} and Table~\ref{Tab:Example_fdd} show the corresponding result. 
\begin{figure}[t]
\begin{center}
\graphicspath{{Simulations/}}
 \def\svgwidth{15cm}
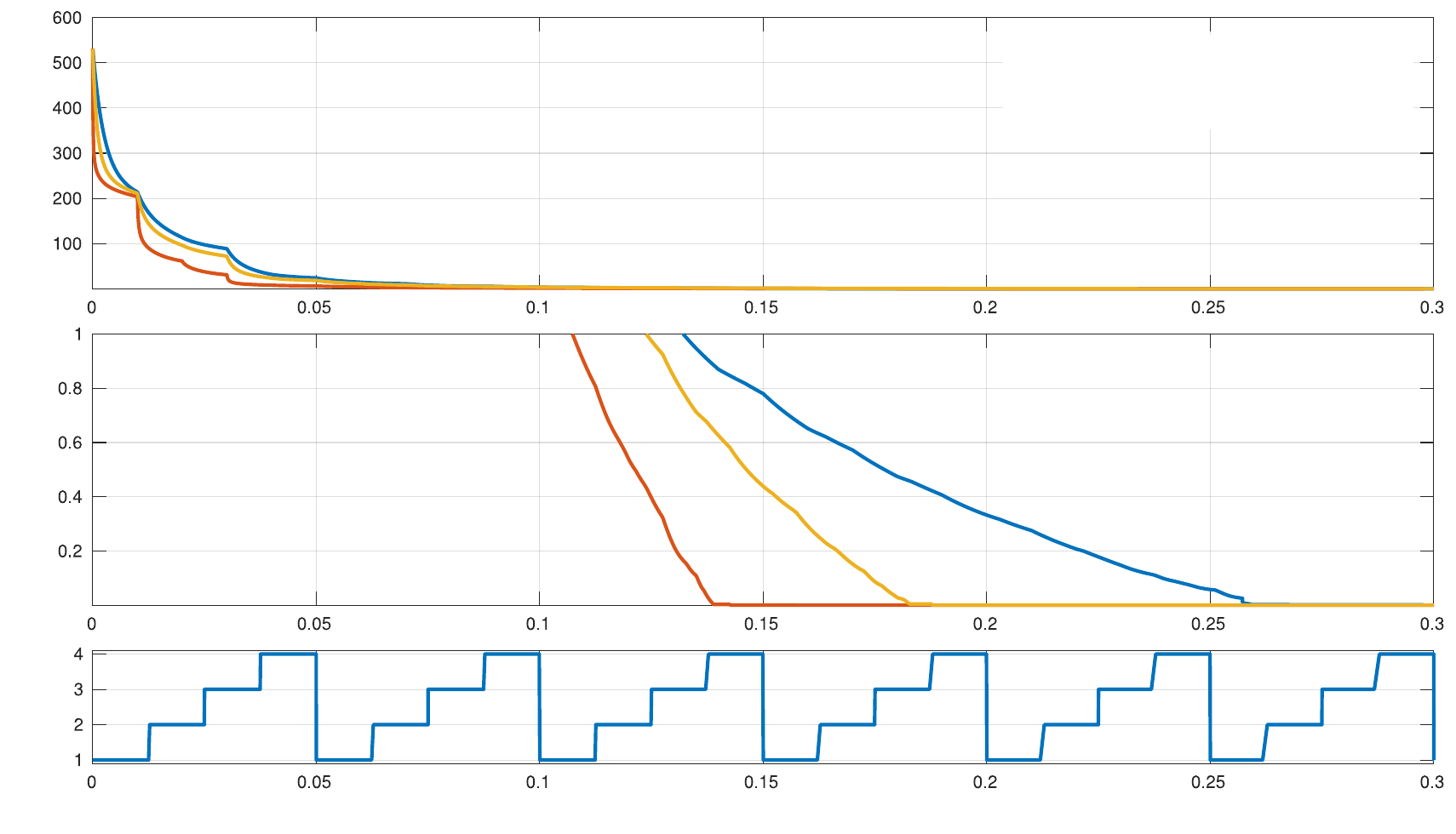
\end{center}
\caption{Convergence of the consensus algorithm for Example \ref{Example_comp_fdd} and the corresponding switching signal.}\label{Fig:fdd_comp}
\end{figure}

\begin{table}
    \centering
    \begin{tabular}{|c|c|c|}
    \hline
        Parameters & Convergence time for protocol~\eqref{Eq:ProtocolA} & Convergence time for protocol~\eqref{Eq:ProtocolB}  \\
        \hline
        $p = 0.1, q = 0.9, k=1.0$ & 0.138s & 0.105s\\
         \hline
         $p = 0.1, q = 1.9, k=0.75$ &0.185s & 0.127s\\
         \hline
         $p = 1.5, q = 12, k=0.1$ & 0.258s & 0.212s\\
         \hline
    \end{tabular}
    \caption{Comparison of convergence time for protocol~\eqref{Eq:ProtocolA} and protocol ~\eqref{Eq:ProtocolA} under different parameter selection.}
    \label{Tab:Example_fdd}
\end{table}
\end{example}

\begin{example}
\label{Example_comp_dfd}
Consider the multi-agent system (\ref{Eq:AgentDyn}) composed of $n=10$ agents with external perturbation $d_i(t) = \sin(40t + 0.1i)$.. The collection of communication topologies, given in Figure~\ref{fig:net1}-\ref{fig:net4} are undirected. The switched dynamic network evolves according to the switching signal $\sigma(t)$ given in Figure~\ref{Fig:fdd_comp} which satisfies the minimum dwell time condition. The initial conditions of the agents are randomly generated and are as follows:
$$x(t_0)=[72.31  167.49 -226.30,   45.68,  246.20, -121.78, -196.90, -128.59,  -88.57,   29.05]^T.$$
Protocol~\eqref{Eq:ProtocolB} with $T_c=1$, $\alpha = 1, \beta = 2, \{\kappa_1,\dots,\kappa_4\}=\{241.5668,  54.2777,  253.3879,   70.0029\},\kappa = 54.2777$ and $\zeta = 0.3693$ and the parameter selection described in Table~\ref{Tab:Example_fdd} achieves consensus in fixed-time. Figure~\ref{Fig:dfd_comp} and Table~\ref{Tab:Example_fdd} show the corresponding result.

 \begin{figure}[t]
\begin{center}
\graphicspath{{Simulations/}}
 \def\svgwidth{15cm}
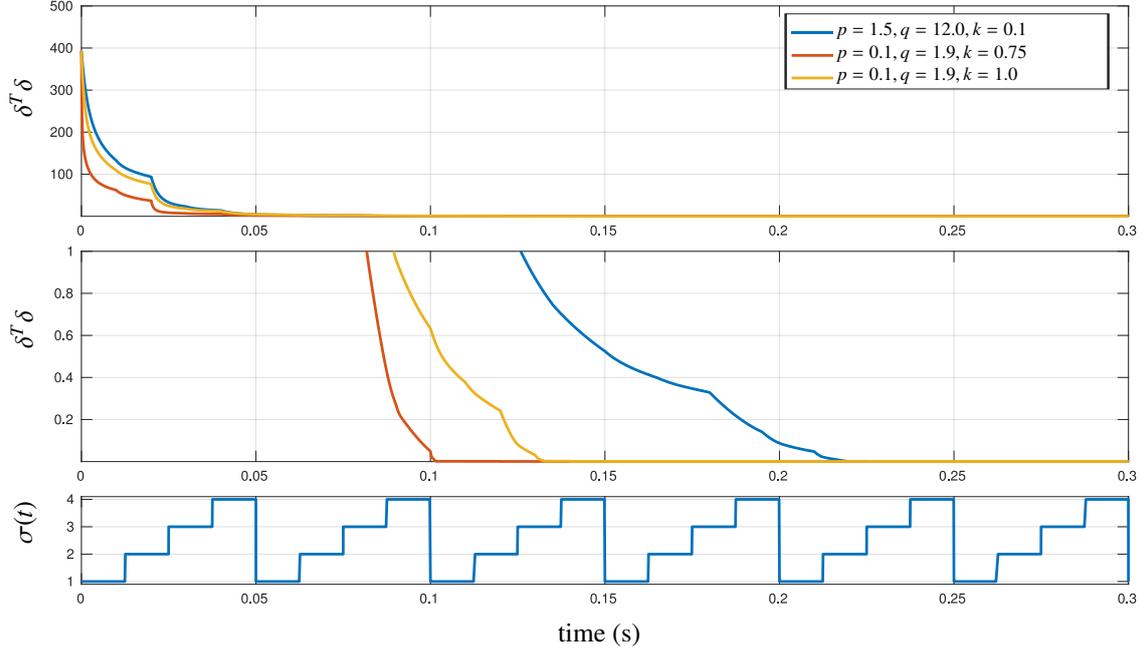
\end{center}
\caption{Convergence of the consensus algorithm for Example \ref{Example_comp_dfd} and switching signal}\label{Fig:dfd_comp}
\end{figure}

\end{example}
 
\section{Conclusion}
\label{Sec:Concl}
In this paper, we presented two robust consensus protocols for perturbed agents with single integrator dynamics. The first one is the one that is computationally simpler; we show that it presents predefined-time convergence under static networks and fixed-time convergence under switched dynamic networks. The second one, in the absence of disturbances, converges to a consensus state that is the average of the agents' initial conditions. This algorithm is shown to be a robust predefined-time consensus algorithm for static and dynamic networks arbitrarily switching among connected graphs. The effectiveness of our approach is shown in simulations where we show that our approach subsumes existing fixed-time consensus protocols and provides more flexibility in the parameter selection, for instance, to improve the convergence. Numerical results were given to illustrate the effectiveness and advantages of the proposed approach and qualitative comparisons were presented to expose the contribution.

\appendix
\section{Some useful inequalities}
\begin{lemma}\cite[Formula 5]{jensen1906}
\label{Th:Jensen}
Let $f(x)$ be a real-valued function and $\lambda_1,\dots,\lambda_n\in[0,1]$ satisfy $\sum_{i=1}^n \lambda_i = 1$. Therefore, $f(x)$ is a convex function if and only if it satisfies Jensen's inequality:
\begin{equation}
\label{eq:jensen}
f\left(\sum_{i=1}^n \lambda_i x_i\right) \leq \sum_{i=1}^n \lambda_if(x_i)
\end{equation}
\end{lemma}

\begin{lemma}
\label{Th:monotone}
Let 
\begin{equation}
\label{eq:poly_func}
f(x) = x(\alpha x^p + \beta x^q)^k
\end{equation}
for $\alpha,\beta,p,q,k>0$, $pk<1$ and $qk>1$. 
Then, $f(x)$ is monotonically increasing for $x>0$. 
\end{lemma}
\begin{proof}
The direct differentiation of $f(x)$ results in:
$$
\frac{d}{dx}f(x) = (\alpha x^p + \beta x^q)^k + kx(\alpha x^p + \beta x^q)^{k-1}(\alpha p x^{p-1} + \beta  x^{q-1})\geq 0
$$
for $x>0$. Therefore, $f(x)$ is monotonically increasing.
\end{proof}

\begin{lemma}
\label{Th:convex_fundamental}
Let the function $f(x;k)$ defined as in \eqref{eq:poly_func}. Then, $f(x)$ is a convex function for $x> 0$.
\end{lemma}
\begin{proof}
Taking the second derivative of $f(x)$ with respect to $x$ leads to:
\begin{align*}
\frac{d^2}{dx^2}f(x) =&\frac{k}{x}(\alpha x^p + \beta x^q)^{k-2}\left( \alpha^2p(kp+1)x^{2p} + ab(2kpq+p+q+(q-p)^2)x^{p+q} + b^2q(kq+1)x^{2q} \right).
\end{align*}
Since $\frac{d^2}{dx^2}f(x)> 0\ \forall x>0$, $f(x)$ is convex for $x>0$. 
\end{proof}

\begin{lemma}
\label{Th:convex}
Let $n\in N$. If $a = (a_1, \dots , a_n)$ is a sequence
of positive numbers, then the following inequality is satisfied
\begin{equation}
\label{eq:poly_jensen}
\frac{1}{n}\sum_{i=1}^n a_i(\alpha a_i^p + \beta a_i^q)^k \geq \left(\frac{1}{n}\sum_{i=1}^na_i\right)\left(\alpha\left(\frac{1}{n}\sum_{i=1}^na_i\right)^p + \beta\left(\frac{1}{n}\sum_{i=1}^na_i\right)^q\right)^k
\end{equation}
for $\alpha,\beta,p,q,k>0$, $pk<1$ and $qk>1$.
\end{lemma}
\begin{proof}
First, recognize $\frac{1}{n}\sum_{i=1}^n a_i$ as a convex combination of the elements of the sequence $a$. Moreover, since $f(x;k)$ defined in \eqref{eq:poly_func} is convex due to Lemma \ref{Th:convex_fundamental}, then it satisfies Jensen's inequality \eqref{eq:jensen}. Evaluating $f\left(\frac{1}{n}\sum_{i=1}^n a_i\right)$ in Jensen's inequality results in expression \eqref{eq:poly_jensen}.
 \end{proof}

\begin{lemma}\cite[Property A.6.1]{Basile1992}
\label{Th:Norm}
Let $z=[z_1 \ \cdots \ z_n]^T\in\mathbb{R}^n$ and 
$$
\Vert z \Vert_p = \left(\sum_{i=1}^n|z_i|^p\right)^\frac{1}{p}
$$
then,
$$
\Vert z \Vert_l \leq \Vert z \Vert_r
$$
for $l\geq r$.
\end{lemma}




\end{document}